\documentclass[10pt]{article}
\usepackage{graphicx}
\usepackage{amsmath}
\usepackage{amssymb}
\usepackage{mathtools}
\usepackage[all]{xy}
\usepackage{tikz-cd}
\usepackage{tikz}
\usepackage{fullpage}
\usepackage{amsthm}

\newtheorem{theorem}{Theorem}[section]
\newtheorem{proposition}[theorem]{Proposition}
\newtheorem{corollary}[theorem]{Corollary}
\newtheorem{remark}[theorem]{Remark}
\newtheorem{lemma}[theorem]{Lemma}

\theoremstyle{definition}
\newtheorem{definition}[theorem]{Definition}
\newtheorem{example}[theorem]{Example}
\newtheorem{exercise}[theorem]{Exercise}
\newtheorem{question}[theorem]{Question}

\newcommand{\calO}{\mathcal{O}}
\newcommand{\C}{\mathbb{C}}
\newcommand{\kbar}{\overline{k}}
\newcommand{\Z}{\mathbb{Z}}
\newcommand{\A}{\mathbb{A}}
\newcommand{\FF}{\mathbb{F}}
\newcommand{\PP}{\mathbb{P}}
\newcommand{\QQ}{\mathbb{Q}}

\newcommand{\R}{\mathbb{R}}
\newcommand{\glob}{\mathcal{O}}

\newcommand{\characteristic}{\operatorname{char}}
\newcommand{\Gr}{\operatorname{Gr}}
\newcommand{\Sym}{\operatorname{Sym}}
\newcommand{\Spec}{\operatorname{Spec}}
\newcommand{\GW}{\operatorname{GW}}
\newcommand{\Tr}{\operatorname{Tr}}
\newcommand{\Hom}{\operatorname{Hom}}
\newcommand{\disc}{\operatorname{disc}}
\newcommand{\grad}{\operatorname{grad}}
\newcommand{\type}{\operatorname{type}}
\newcommand{\ev}{\operatorname{ev}}
\newcommand{\gen}{\text{generic}}
\newcommand{\Nodes}{\operatorname{Nodes}}
\newcommand{\Hess}{\operatorname{Hess}}
\newcommand{\lra}[1]{\langle #1 \rangle}

\usepackage[
textwidth=3cm,
textsize=small,
colorinlistoftodos]
{todonotes}

\begin{document}
\title{Applications to $\A^1$-enumerative geometry of the $\A^1$-degree
}


\author{Sabrina Pauli         \and
        Kirsten Wickelgren 
}
\date{}



\maketitle

\begin{abstract}
These are lecture notes from the conference Arithmetic Topology at the Pacific Institute of Mathematical Sciences on applications of Morel's $\A^1$-degree to questions in enumerative geometry. Additionally, we give a new dynamic interpretation of the $\A^1$-Milnor number inspired by the first named author's enrichment of dynamic intersection numbers.  
\end{abstract}

\section{Introduction}
\label{intro}
$\A^1$-homotopy theory provides a powerful framework to apply tools from algebraic topology to schemes. In these notes, we discuss Morel's \emph{$\A^1$-degree}, giving the analog of the Brouwer degree in classical topology, and applications to enumerative geometry. Instead of the integers, the $\A^1$-degree takes values in bilinear forms, or more precisely, in the Grothendieck-Witt ring $\GW(k)$ of a field $k$, defined to be the group completion of isomorphism classes of symmetric, non-degenerate bilinear $k$-forms. This can result in an enumeration of algebro-geometric objects valued in $\GW(k)$, giving an $\A^1$-enumerative geometry over non-algebraically closed fields. One recovers classical counts over $\C$ using the rank homomorphism $\GW(k) \to \Z$, but $\GW(k)$ can contain more information. This information can record arithmetic-geometric properties of the objects being enumerated over field extensions of $k$. 

In more detail, we start with the classical Brouwer degree. We introduce enough $\A^1$-homotopy theory to describe Morel's degree and use the Eisenbud-Khimshiashvili-Levine signature formula to give context for the degree and a formula for the local $\A^1$-degree.  The latter is from joint work of Jesse Kass and the second-named author. A point of view on the classical Euler number is as a sum of local degrees. This in turn gives a point of view on an $\A^1$-Euler number \cite{2017arXiv170801175K} and enrichments of enumerative results. We give some due to Tom Bachmann, Jesse Kass, Hannah Larson, Marc Levine, Stephen McKean, Padma Srinivasan, Isabel Vogt, Matthias Wendt, and the authors. in Section \ref{section: Examples of enriched counts}. We describe joint work of Kass and the second named author on $\A^1$-Milnor numbers in Section \ref{section:A1Milnornumbers}.

Inspired by the first named author's enriched theory of dynamic intersection, we then give a new interpretation of the $\A^1$-Milnor number. See Section \ref{subsection: dynamic interpretation}, Theorems \ref{A1-Milnor_number_sum_A1-Milnor_numbers_bifurcations} and \ref{A1-Milnor_number_sum_node-types}.

Finally, we discuss joint work in progress of Kass, Levine, Jake Solomon and the second named author on the degree of a map of smooth schemes (as opposed to of a map between $\A^1$-spheres) and counts rational curves plane curves of degree $d$ through $3d-1$ points.

\section{Motivation from classical homotopy theory}
\label{section: Motivation from classical homotopy theory}
\subsection{The Brouwer degree}
Let $S^n=\{(x_0,\dots,x_n)\in\mathbb{R}^{n+1}: \sum_{i=0}^nx_i=1\}$ be the $n$-sphere. 
Since $S^n$ is \emph{orientable}, its top homology group $H_n(S^n)$ is isomorphic to $\Z$. Hence, a map $f:S^n\rightarrow S^n$ induces a homomorphism $f_*:\Z\rightarrow \Z$. For a choice of generator $\alpha$ of $H_n(S^n)\cong \Z$ (which is equivalent to choosing an orientation of $S^n$), it follows that $f_*(\alpha)=d\alpha$. The integer $d$ is called the \emph{Brouwer degree} of $f$. Two homotopic maps $f,g:S^n\rightarrow S^n$ have the same Brouwer degree and it turns out that the Brouwer degree establishes an isomorphism between homotopy classes of pointed maps $S^n\rightarrow S^n$ and the integers
\[\deg:[S^n,S^n]\xrightarrow{\cong} \Z.\]
\begin{remark}
Note that $S^n$ is homotopy equivalent to $\PP^n(\mathbb{R})/\PP^{n-1}(\mathbb{R})$. Later in the $\A^1$-homotopy version of the Brouwer degree, $S^n$ will be replaced by the 'quotient' of schemes $\PP^n/\PP^{n-1}$.
\end{remark}
\subsection{The Brouwer degree as a sum of local degrees}
\label{subsection: Brouwer degree as sum of local degrees}

Assume $p\in S^n$ such that $f^{-1}(p)=\{q_1,\dots,q_m\}$. Then the Brouwer degree $\deg f$ can be expressed as a sum of local degrees as follows: Let $V$ be a small ball around $p$ and $U$ a small ball around $q\in \{q_1,\dots,q_m\}$ such that $f^{-1}(p)\cap V=\{q\}$. The quotient spaces $U/(U\setminus \{q\})\simeq U/(U\setminus \partial U)$ and $V/(V\setminus\{p\})\simeq V/\partial V$ are homotopy equivalent to $S^n$. Let $$\bar{f}:S^n\simeq U/(U\setminus \{q\})\rightarrow V/(V\setminus\{p\})\simeq S^n$$ be the map of spheres induced by $f$ under orientation preserving homotopy equivalences $S^n\simeq U/(U\setminus\{q\})$ and $V/(V\setminus \{p\})\simeq S^n$. We define the local degree $\deg_qf$ of $f$ at $q$ to be the Brouwer degree of $\bar{f}$
\[\deg_qf:=\deg \bar{f}.\]
If $p$ is a regular value, then $f$ is a local homeomorphism and $\bar{f}$ is a homeomorphism. It follows that $\deg_qf\in\{\pm 1\}$. More precisely, $\deg_q f$ is $+1$ when $\bar{f}$ is orientation preserving and $-1$ when $\bar{f}$ is orientation reversing. Consequently, it is often easier to compute $\deg f$ as a sum of local degrees, especially because we have the following formula for the local degree from differential topology.

\subsubsection{A formula from differential topology}
\label{subsection: formula from differential topology}
Let $x_1,\dots,x_n$ be oriented coordinates near $q$ and $y_1,\dots,y_n$ be oriented coordinates near $p$.
In these coordinates, $f$ is given by $f=(f_1,\dots,f_n):\R^n\rightarrow \R^n$. Define the \emph{jacobian element} at $q$ by $Jf(q):= \det (\frac{\partial f_i}{\partial x_j})$. Then
\begin{equation}
\label{equation: def jacobian}
\deg_q(f)=\begin{cases*}
      +1 & if $Jf(q)>0$ \\
      -1        & if $Jf(q)<$ 0.
    \end{cases*}
\end{equation}

Later, we will define the local $\A^1$-degree which will record $Jf$ up to multiplication by squares, that is, the image of $Jf$ in $k^{\times}/(k^{\times})^2$ for an arbitrary field $k$. Note that for $k=\R$ this agrees with \eqref{equation: def jacobian} since $\R^{\times}/(\R^{\times})^2\cong\{\pm1\}$.
To be more precise, we first discuss the Grothendieck--Witt ring.

\section{The Grothendieck-Witt ring of $k$}
\label{section: GW(k)}
\subsection{Symmetric bilinear forms}

Let $R$ be a commutative ring and $P$ a finitely generated projective $R$-module.
A \emph{symmetric bilinear form} on $P$ over $R$ is a bilinear map $$b:P\times P\rightarrow R$$ such that $b(u,v)=b(v,u)$ for all $u,v\in P$. 
Let $P^*:=\operatorname{Hom}_R(P,R)$. The form $b$ is \emph{non-degenerate} if for all $u\in P$ the map $P\rightarrow P^*$, $u\mapsto b(-,u)$ is an isomorphism. 

Two symmetric bilinear forms $b_1:P_1\times P_1\rightarrow R$ and $b_2:P_2\times P_2\rightarrow R$ are \emph{isometric} if there is a $R$-linear isomorphism $\phi:P_1\rightarrow P_2$ such that $b_2(\phi(u),\phi(v))=b_1(u,v)$ for all $u,v\in P_1$.
This is an equivalence relation.

The \emph{direct sum} of two (non-degenerate) symmetric bilinear forms $b_1:P_1\times P_1\rightarrow R$ and $b_2:P_2\times P_2\rightarrow R$ is the (non-degenerate) symmetric bilinear form 
\[b_1\oplus b_2:P_1\oplus P_2\rightarrow R \text{, }((x_1,x_2),(y_1,y_2))\mapsto b_1(x_1,y_1)+b_2(x_2,y_2).\]
The \emph{tensor product} of $b_1$ and $b_2$ is the (non-degenerate) symmetric bilinear form
\[b_1\otimes b_2:P_1\otimes P_2\rightarrow R\text{, }((x_1\otimes x_2),(y_1\otimes y_2))\mapsto b_1(x_1,y_1)b_2(x_2,y_2).\]

The set of isometry classes of finite rank non-degenerate symmetric bilinear together with the direct sum $\oplus$ and the tensor product $\otimes$ forms a \emph{semi-ring}.

\subsubsection{Over a field $k$}
If $R=k$ is a field, then $P=V$ is a finite dimensional vector space over $k$. 
We call $n=\dim_k V$ the \emph{rank} of the symmetric bilinear form $b$. For a chosen basis $v_1,\dots,v_n$ of $V$ the associated \emph{Gram matrix} with entries $b(v_i,v_j)$ of $b$ is symmetric. 
Any symmetric bilinear form can be \emph{diagonalized} meaning that there exists a basis $v_1,\dots,v_n$ of $V$ such that the Gram matrix $b(v_i,v_j)$ is diagonal. Furthermore, a 
symmetric bilinear form over $k$ is non-degenerate if and only if the determinant of the Gram matrix is non-zero.

\begin{remark}
For $x\in V$, $q:V\rightarrow k$ defined by $q(x)=b(x,x)$ is a \emph{quadratic form}. Conversely, if $\operatorname{char} k\neq 2$ a quadratic form $q:V\rightarrow k$ gives rise to the symmetric bilinear for $b(x,y)=\frac{1}{2}(q(x+y)-q(x)-q(y))$.
\end{remark}

\subsection{Group completion}
Let $M$ be a \emph{commutative monoid}. The \emph{Grothendieck group} $K(M)$ of $M$ is the abelian group defined by the following universal property: There is a monoid homomorphism $i:M\rightarrow K(M)$ such that for any monoid morphism $m:M\rightarrow A$ to an abelian group $A$ there exists a unique group homomorphism $p:K(M)\rightarrow A$ such that $m=p\circ i$.

\[
\begin{tikzcd}
M\arrow{dr}{i}\arrow{rr}{m}&& A\\
&K(M)\arrow[dotted]{ur}{\exists!p}&
\end{tikzcd}
\]

\begin{example}
The Grothendieck group of the natural numbers $\mathbb{N}_0$ is the integers $\Z$
$$K(\mathbb{N}_0)=\Z.$$
\end{example}

There are several explicit constructions of the Grothendieck group (see for example \cite{MR3076731}).

\subsection{$\GW(R)$}
Let $R$ be a commutative ring. 
\begin{definition}
The \emph{Grothendieck-Witt ring $\GW(R)$} of $R$ is the group completion, i.e. the Grothendieck group, of the semi-ring of isometry classes of non-degenerate symmetric bilinear forms over $R$.
\end{definition}

\subsubsection{Over a field $k$}
Since over a field $k$ any symmetric bilinear form can be diagonalized, we can describe $\GW(k)$ in terms of explicit generators and relations.
Let $\langle a\rangle $ represent the 1-dimensional non-degenerate symmetric bilinear form $k\times k\rightarrow k$ defined by $(x,y)\mapsto axy$ for $a\in k^{\times}$ a unit in $k$. Then $\GW(k)$ is generated by $\langle a\rangle$ for $a\in k^{\times}$ subject to the following relations
\begin{enumerate}
\item $\langle a\rangle =\langle ab^2\rangle $ for $a,b\in k^{\times}$
\item $\langle a\rangle \langle b\rangle =\langle ab\rangle $ for $a,b\in k^{\times}$
\item $\langle a\rangle +\langle b\rangle =\langle a+b\rangle +\langle ab(a+b)\rangle $ for $a,b\in k^{\times}$ and $a+b\neq 0$
\item $\langle -a\rangle +\langle a\rangle =\langle -1\rangle +\langle 1\rangle $ for $a\in k^{\times}$.
\end{enumerate}
\begin{remark}
1.-3. imply 4. 

However, to simplify computations, we add the fourth relation and call $\langle 1\rangle +\langle -1\rangle $ the \emph{hyperbolic form}.
\end{remark}

\subsubsection{Examples}
\begin{example} 
For an algebraically closed field like the complex numbers $\C$, it follows from the first relation that any element of the Grothendieck-Witt ring is equal to the sum of $\langle 1\rangle's$. Hence, the rank establishes an isomorphism $\GW(\C)\cong \Z$.
\end{example}
\begin{example} $\GW(\R)\cong \Z\times \Z$
\end{example}
\begin{proof}
Let $V$ be an $n$-dimensional $\R$-vector space and $b:V\times V\rightarrow \R$ a non-degenerate symmetric bilinear form.
By Silvester's theorem
 there is a basis $\{v_1,\dots,v_n\}$ of $V$ such that the Gram Matrix $(b(v_i,v_j))_{i,j}$ is of the form
\[
\begin{pmatrix}
1 &  &&&&  \\
 & \dots &&&&\\
 & &1&&&\\
 & & & -1&&\\
 & & & & \dots&\\
 & & & & & -1
\end{pmatrix}.
\]
Let the \emph{signature} $\operatorname{sgn}(b)$ of $b$ be equal to number of $1$'s minus the number of $-1$'s.
Then $\GW(\R)\cong \{(r,s)\in \Z\times\Z: r+s\equiv 0 \operatorname{mod} 2\}\cong \Z\times \Z$ where $r$ is the rank and $s$ the signature of the bilinear form.
\end{proof}
\begin{example}
$\GW(\mathbb{F}_q)\cong \Z\times \mathbb{F}_q^{\times}/(\mathbb{F}_q^{\times})^2$ where the isomorphism is given by the rank and discriminant (= determinant of the Gram matrix).
\end{example}

\begin{example}
Let $k$ be a field. Then $\GW(k[t])\cong \GW(k)$ by Harder's Theorem (see \cite[Theorem 3.13, Chapter VII]{MR2235330} for $\operatorname{char} k\neq 2$ and \cite[Lemma 30]{MR3909901} for $\operatorname{char} k=2$) 
\end{example}

\begin{example}
Let again $k$ be a field, and for simplicity assume that the characteristic of $k$ is not $2$. Then by Springer's Theorem \cite[Theorem 1.4, Chapter VI]{MR2104929}
\[\frac{\GW(k)\oplus\GW(k)}{\Z(\langle1\rangle+\langle-1\rangle,-(\langle1\rangle+\langle-1\rangle))}\xrightarrow{\cong}\GW(k((t))),\text{ }(\langle u\rangle,\langle v\rangle)\mapsto\langle u\rangle+\langle tv\rangle\]
is an isomorphism.
\end{example}

\begin{example}
As in the previous example, let $k$ be a field of characteristic not $2$. Then the extension $k \subset k[[t]]$ defines an isomorphism $\GW(k[[t]])\cong \GW(k)$. In more detail, $\GW(k[[t]])$ is the kernel of the second residue homomorphism $\GW(k((t)) ) \to \GW(k)$ associated to the ideal $(t)$ \cite[Theorem C]{MR1670591}.
\end{example}

\begin{example}
Let $k$ be a field of characteristic not $2$. The kernel $I$ of the rank map $\operatorname{rk}:\GW(k)\rightarrow \Z$ is called the \emph{fundamental ideal}.
The Milnor conjectures \cite{Milnor_AlgK-theory_quadratic_forms} state that 
\[I^n/I^{n+1}\cong K^M_n(k)\otimes \Z/2\cong H^n_{\text{\'{e}t}}(k;\Z/2)\]
and was proven by Orlov--Vishik--Voevodsky \cite{MR2276765} and Voevodsky \cite{MR2031199}, \cite{MR2031198}, respectively. One can interpret such isomorphisms as giving invariants of bilinear forms (in $I^n$) valued in Milnor K-theory or \'etale cohomology. The first of these invariants are the rank, discriminant, Hasse-Witt and Arason invariants. For fields of finite \'etale cohomological dimension, this gives a finite list of invariants capable of showing two sums/differences of generators are the same or distinguishing between them. 
\end{example}

\subsubsection{A transfer map}
Let $k\subset L$ a separable field extension. The \emph{transfer} of a non-degenerate symmetric bilinear form $b:V\times V\rightarrow L$ is the form over $k$
\[V\times V\xrightarrow{b}L\xrightarrow{\Tr_{L/k}}k\]
where $\Tr_{L/k}$ denotes the field trace, equal to the sum of the Galois conjugates.
This yields a homomorphism
\[\operatorname{Tr}_{L/k}:\GW(L)\rightarrow \GW(k).\]

For example, $\Tr_{L/k} \langle 1 \rangle$ is the usual class of the trace form of the field extension from number theory.

\section{$\A^1$-homotopy theory and degree}
\label{section: A1 homotopy theory}
Instead of remembering only the sign of $Jf(q)$ in \eqref{equation: def jacobian}, it is an idea of Lannes and Morel to remember the class $\langle Jf(q)\rangle$ in $\GW(k)$, that is $Jf(q)$ up to squares, and get a count in the Grothendieck-Witt ring $\GW(k)$ instead of the integers $\Z$. 
\subsection{The degree of an endomorphism of $\PP^1$}
\label{subsection:degree_endo_P1}
As a first case, consider endomorphisms of the projective line $\PP^1$.
Let $f:\PP^1_k\rightarrow\PP^1_k$, $p\in \PP^1(k)$ and $f^{-1}(p)=\{q_1,\dots,q_m\}$. Suppose $Jf(q_i)=f'(q_i)\neq0$ for all $i=1,\dots,m$ and define 
\[\deg f:= \sum_{i=1}^m\langle Jf(q_i)\rangle \in \GW(k).\]
This does not depend on $p$.

\begin{exercise}
\begin{enumerate}
\item $\deg^{\A^1} (\PP^1_k\rightarrow \PP^1_k, z\mapsto az)=\langle a\rangle \in\GW(k)$
\item $\deg^{\A^1} (\PP^1_k\rightarrow \PP^1_k, z\mapsto z^2)=\langle 1\rangle +\langle -1\rangle \in\GW(k)$
\end{enumerate}
\end{exercise}

\begin{remark}
Naively one can define an $\A^1$-homotopy between two morphisms $f:X\rightarrow Y$ and $g:X\rightarrow Y$ as a morphism $X\times\A^1\rightarrow Y$ which equals to $f$ (respectively $g$) when restricted to $0\in \A^1$ (respectively $1\in \A^1$). Functions are said to be in the same \emph{naive pointed homotopy class} if they are equivalent under the equivalence relation generated by naive $\A^1$-homotopy.

In \cite{MR3059240} C. Cazanave finds a monoid structure on the set of naive pointed homotopy classes of morphisms $\PP^1\rightarrow \PP^1$ (where $\PP^1$ is pointed at infinity) and shows that the group completion of this monoid equals the $\A^1$-homotopy classes of pointed maps from $\PP^1$ to $\PP^1$ in the sense of Morel-Voevodsky which we define in the next subsection \ref{subsection: the homotopy category}. He furthermore provides an explicit formula for the degree of $f: \PP^1 \to \PP^1$: the degree $\deg f$ is given by a certain B\'ezout form \cite[Definition 3.4]{MR3059240}, yielding another explicit and computationally tractable method to compute $\deg f.$
\end{remark}

We can do this in higher dimensions as well. Just as in classical topology, $\PP^n/\PP^{n-1}$ is a `sphere' in $\A^1$-homotopy theory. Morel's $\A^1$-degree homomorphism
\begin{equation}
\label{eq: Morel degree}
\deg^{\A^1}:[\PP^n/\PP^{n-1},\PP^n/\PP^{n-1}]_{\A^1}\rightarrow\GW(k)
\end{equation} assigns an element of $\GW(k)$ to each $\A^1$-homotopy class of a morphisms $\PP^n/\PP^{n-1}\rightarrow\PP^n/\PP^{n-1}$ \cite{MR2934577}. In order to understand this degree \eqref{eq: Morel degree}, we first have to make sense of $\PP^n/\PP^{n-1}$. Morel and Voevodsky's $\A^1$-homotopy theory allows this and much more.

\subsection{The homotopy category $ho(Spc_k)$}
\label{subsection: the homotopy category}
We give a brief sketch of $\A^1$-homotopy theory \cite{MR1813224} here. Further exposition can be found in  \cite{MR3727503} \cite {MR2465990} \cite{2019arXiv190208857W}, for example. 

$\PP^n/\PP^{n-1}$ should be the colimit of the diagram 
\[
\begin{tikzcd}
\PP^{n-1}\arrow{r}\arrow{d}&\PP^n\\
*.
\end{tikzcd}
\]
However, the category of (smooth) schemes over $k$ in not closed under taking colimits and we need to enlarge it.

Let $Sm_k$ be the category of smooth (separated of finite type) schemes over a field $k$. We embed $Sm_k$ fully faithfully into the category of simplicial presheaves $sPre(Sm_k)$, i.e., functors $Sm_k^{\operatorname{op}}\rightarrow sSet$, via the Yoneda embedding
\[Sm_k\rightarrow sPre(Sm_k), X\mapsto \operatorname{Hom}_{Sm_k}(-,X).\]
The category $sPre(Sm_k)$ has finite limits and colimits and the quotient $\PP^n/\PP^{n-1}$ is an object in this category. Note that the category $sSet$ of simplicial sets also embedds into $sPre(Sm_k)$ via the constant embedding
\[sSet\rightarrow sPre(Sm_k), T\mapsto ((-)\mapsto T).\] 

The category $sPre(Sm_k)$ can be given the structure of a simplicial model category \cite{MR1944041} or can be viewed as an $\infty$-category \cite{MR2522659}. Here, we will think of both as structures which encode homotopy theories, and blur the (important and interesting) differences between them. In both viewpoints, there is a notion of weak equivalence and there is a well-defined homotopy category, which is the category where all the weak equivalences are inverted. In either setting, one can use \emph{Bousfield localization} (see \cite{MR1944041}) to impose additional weak equivalences or equivalently invert more morphisms in the homotopy category. 

In a certain technical sense, $sPre(Sm_k)$ is obtained by freely adding colimits. However, colimits corresponding to gluing ``open covers" already existed in $Sm_k$. We wanted these, but destroyed them in passing to $sPre(Sm_k)$. To rectify the situation, one uses Bousfield localization to impose the condition that a map from an open cover of $X$ to $X$ is a weak equivalence.

The ``open covers" we mean in this context are those associated to the \emph{Nisnevich} \emph{Grothendieck topology}. (See e.g. \cite{MR2334212} for more information on Grothendieck topologies). The Nisnevich topology is finer than the Zariski topology but coarser than the \'{e}tale topology and carries useful properties of both of them. 
It is the Grothendieck topology on $Sm_k$ generated by \emph{elementary distinguished squares}, that is Cartesian squares in $Sm_k$
\[\begin{tikzcd}V\arrow[r, ]\arrow[d]& Y \arrow[d, "p"] \\U\arrow[r,"i"]& X\end{tikzcd}\]
such that $i$ is an open immersion, $p$ is \'{e}tale and $p^{-1}(X\setminus U)_{\operatorname{red}}\rightarrow (X\setminus U)_{\operatorname{red}}$ is an isomorphism. Associated to an open cover of a smooth scheme $X$, we have a simplicial presheaf corresponding to its \v{C}ech nerve. Let $L_{\operatorname{Nis}}$ denote the Bousfield localization requiring all such maps to be weak equivalences. $L_{Nis}$ can be thought of as a functor $$L_{\operatorname{Nis}}:sPre(Sm_k)\rightarrow Sh_k$$ whose target $Sh_k$ is a homotopy theory of \emph{sheaves}.

In $\A^1$-homotopy theory, one wants $\A^1$ to play the role of the unit inverval $[0,1]$ in classical topology. So we force $\A^1$ to be contractible, meaning it is weakly equivalent to the point. In order for the product structure to have desirable properties, we moreover force $X\times \A^1\rightarrow X$ to be a weak equivalence for all smooth schemes $X$, and let  $L_{\A^1}:Sh_k\rightarrow Spc_k$ denote the resulting Bousfield localization. We call the resulting homotopy theory $Spc_k$ \emph{spaces} over $k$. The total process can be summarized:

\[Sm_k\rightarrow sPre(Sm_k)\xrightarrow{L_{\operatorname{Nis}}} Sh_k\xrightarrow{L_{\A^1}} Spc_k\]

\noindent Let $[-,-]_{\A^1}$ denote the maps in the homotopy category $ho(Spc_k)$ of $Spc_k$. 

Having sketched $\A^1$-homotopy theory, the codomain of Morel's degree map has been defined, and we state:
\begin{theorem}(Morel)
The degree map $\deg^{\A^1}:[\PP^n/\PP^{n-1},\PP^n/\PP^{n-1}]_{\A^1}\rightarrow\GW(k)$
is an isomorphism for $n\ge 2$ \cite{MR2934577}.
\end{theorem}

Moreover, Morel's degree extends the topological degree in the sense that the following diagram is commutative: $$\xymatrix{ [S^n,S^n] \ar[d]^{\deg} &&& \ar[lll]^{ \R\textrm{-points}}[\PP^n_k/\PP_k^{n-1}, \PP^n_k/\PP_k^{n-1}]_{\A^1} \ar[d]^{\deg} \ar[rrr]^{ \C\textrm{-points}}&&& [S^{2n},S^{2n}] \ar[d]^{\deg} \\
\Z &&& \ar[lll]^{\operatorname{signature}} \GW(k) \ar[rrr]_{\operatorname{rank}} &&& \Z} $$ for any subfield $k$ of $\R$.

\subsection{Purity}
\label{subsection: purity}
Let $V\rightarrow X$ be a vector bundle and $i:X\hookrightarrow V$ the zero section.
The \emph{Thom space} of $V$ is defined as follows
\[Th(V):=V/(V\setminus i(X)).\]
In the $\A^1$-homotopy category $ho(Spc_k)$ the Thom space $Th(V)$ is isomorphic to $\PP(V\oplus \glob)/\PP(V)$ where $\glob\rightarrow X$ is the trivial rank 1 bundle \cite[Proposition III.2.17]{MR1813224}.
\begin{theorem}[Homotopy purity]
\label{thm: purity}
Let $Z\hookrightarrow X$ be a closed immersion in $Sm_k$ and $N_ZX\rightarrow Z$ its normal bundle. Then
\[X/(X\setminus Z)\cong Th(N_ZX)\]
is a canonical isomorphism in $ho(Spc_k)$ \cite[Theorem III.2.23]{MR1813224}.
\end{theorem}

\section{The local $\A^1$-degree}
\label{subsection: local A1-degree}

In Section \ref{subsection: Brouwer degree as sum of local degrees}, we discussed the local topological Brouwer degree. There is an analogous local $\A^1$-degree. We came across it already in Section \ref{subsection:degree_endo_P1} to give the degree of an endomorphism of $\PP^1$, without introducing it in its own right. We do this now.

Suppose $f$ is a morphism $f:\A^n\rightarrow \A^n$ and $x$ in $\A^n(k)$ is such that $x$ is isolated in $f^{-1}(f(x))$, i.e., there is a Zariski open set $U\subset \A^n$ with $x\in U$ such that  $f^{-1}(f(x))\cap U=\{x\}$. Then by the homotopy purity theorem \ref{thm: purity} it follows that $U/(U\setminus \{x\})$ is canonically isomorphic to the Thom space $Th(N_x\A^n)$ which is canonically isomorphic to $\PP(N_x\A^n\oplus \glob)/\PP(N_x\A^n)$ in the $\A^1$-homotopy category $ho(Spc_k)$. The choice of basis for $N_x\A^n$ determines an isomorphism $\PP(N_x\A^n\oplus \glob)/\PP(N_x\A^n)\simeq\PP^n_k/\PP^{n-1}_k$ in $ho(Spc_k)$, and the canonical trivialization of the tangent bundle of affine space $\A^n$ gives a preferred choice.

The \emph{local $\A^1$-degree} $\deg^{\A^1}_xf$ of $f$ at $x$ is defined to be the degree of 
\[\PP_k^n/\PP_k^{n-1}\cong Th(N_x\A^n)\cong U/(U\setminus \{x\})\xrightarrow{\bar{f}} \A^n/(\A^n\setminus \{f(x)\})\cong \PP_k^n/\PP_k^{n-1}. \]

As before let $Jf:=\det \frac{\partial f_i}{\partial x_j}$ be the \emph{jacobian element}.
\begin{example}
Let $x\in \A^n$ be a zero of $f$. If $x$ is $k$-rational and $Jf(x)\neq 0$ in $k$, then $\deg^{\A^1}_xf=\langle Jf(x)\rangle \in \GW(k)$ \cite{MR3909901}.
\end{example}
\begin{example}
\label{example: local A1 degree separable field extension}
Let $x\in \A^n$ be a zero of $f$. 
Assume $x$ is defined over a separable field extension $k(x)/k$ and $Jf(x)\neq 0$ in $k(x)$, then there is an extension of the definition of local degree and it can be computed to be $\deg^{\A^1}_xf=\operatorname{Tr}_{k(x)/k}\langle Jf(x)\rangle \in \GW(k)$ \cite[Proposition 15]{MR3909901}.
\end{example}

\subsection{The Eisenbud-Levine/Khimshiashvili signature formula}
\label{section: EKL-form}
When $x\in \A^n_k$ is a non-simple isolated zero of $f:\A^n_k\rightarrow \A^n_k$, i.e., $Jf(x)=0$, we can compute $\deg^{\A^1}_xf$ as the \emph{Eisenbud-Levine/Khimshiashvili form}, short \emph{EKL-form}. This form is named after the \emph{Eisenbud-Levine/Khimshiashvili signature formula}: For $k=\R$ Eisenbud-Levine and Khimshiashvili, independently, defined a non-degenerate symmetric bilinear form, the EKL-form over $\R$ whose signature is equal to the local topological Brouwer degree \cite{eisenbud77} \cite{khimshiashvili}. This form is defined on the vector space $\frac{\R[x_1,\dots,x_n]_x}{(f_1,\ldots, f_n)}$. For $k = \C$, the dimension of this vector space was shown to be the local topological Brouwer degree $\deg_x f$ by Palamodov in \cite[Corollary 4]{MR0236424}.

The EKL-form is defined in purely algebraic terms, and can thus be defined over any field $k$. Eisenbud raised the question if there was an interpretation of the EKL-form over an arbitrary field \cite[p. 163-4 some remaining questions (3)]{MR494226}. The answer is yes: In \cite{MR3909901} Kass and the second named author show that the class of the EKL-form in $\GW(k)$ is equal to the local $\A^1$-degree when $k=k(x)$ and Brazelton, Burklund, McKean, Montoro and Opie extend this result to separable field extensions $k(x)/k$ \cite{brazelton2019trace}.
\begin{theorem}
We have 
\[\deg^{\A^1}_xf=\omega^{EKL}\]
in $\GW(k)$.
\end{theorem}

We recall the definition of the EKL-form from \cite{MR3909901}.
When $x\in \A^n_k$ is an isolated zero of $f:\A^n_k\rightarrow \A^n_k$, the local algebra $\glob_{f^{-1}(0),x}$ is a finite dimensional $k$-vector space. 
\begin{definition}
Assume $\operatorname{char} k$ does not divide the rank of $\glob_{f^{-1}(0),x}$.
Then the \emph{EKL-form} is given by
\[\omega^{EKL}:\glob_{f^{-1}(0),x}\times \glob_{f^{-1}(0),x}\rightarrow k, (a,b)\mapsto \eta(ab)\]
where $\eta:\glob_{f^{-1}(0),x}\rightarrow k$ is any $k$-linear map with $\eta(Jf)=\dim_k\glob_{f^{-1}(0),x}$ where $Jf=\frac{\partial f_i}{\partial x_j}$ is the jacobian element.
\end{definition}
The EKL-form is well-defined, i.e., it does not depend on the choice of $\eta$ and is non-degenerate \cite[Lemma 6]{MR3909901}.
\begin{remark}
The EKL-form can also be defined when $\operatorname{char} k$ divides the rank of $\glob_{f^{-1}(0),x}$ in terms of the `distinguished socle element' $E$ \cite[$\S1$]{MR3909901}.

To define $E$, one needs `Nisnevich coordinates' which always exist over a field \cite[$\S1$]{MR3909901} and \cite[Definition 17]{2017arXiv170801175K}.
\end{remark}

\begin{example}\label{EKLformz2}
Let $f:\R\rightarrow \R$ be defined by $f(z)=z^2$. Then $Jf=2z$ and $(1, 2z)$ is a basis for $\glob_{f^{-1}(0),0}=\frac{\mathbb{R}[z]_{(z)}}{(z^2)}$. Choose $\eta$ such that $\eta(1)=0$ and $\eta(2z)=2$. Then $\omega^{EKL}$ is the rank two form defined by the matrix 
\[\begin{bmatrix}
0 & 2\\
2 & 0
\end{bmatrix}\]
that is the hyperbolic form $\langle1\rangle+\langle-1\rangle\in \GW(\mathbb{R})$. The signature of $\omega^{EKL}(f)$ is 0 which agrees with the intuition: For $0\neq a\in \R$, the preimage $f^{-1}(a)$ is either empty (when $a<0$) or consists of 2 points (when $a>0$). Locally around one of these points, $f$ is orientation preserving, and $f$ is orientation reversing around the other point, contributing a +1 and -1, respectively, to the degree of $f$.
\end{example}

\section{$\A^1$-Milnor numbers}
\label{section:A1Milnornumbers}

\subsection{Milnor numbers over $\C$}
\label{subsection: C Milnor numbers}

The Milnor number is an integer multiplicity associated to an isolated critical point of a polynomial (or more generally a holomorphic) map $f:\C^n \to \C$.\footnote{ A {\em critical point} of $f$ is a point where the partials $\partial_i f$ vanish and a critical point is said to be {\em isolated} if there is an open neighborhood around that point not containing other critical points.} Such critical points $x$ correspond to isolated singularities of the complex hypersurfaces $\{ f = f(x)\}$.\footnote{A hypersurface of affine (respectively projective) space is the zero locus of a (respectively homogenous) polynomial, and a point $x$ on a scheme $X$ is said to be an {\em isolated singularity} if there is a Zariski open neighborhood $U$ of $x$ such that the only singular point of $U$ is $x$.} There are numerous definitions of the Milnor number, which of course are all equal, creating lovely pictures of what this number means. See for example \cite{Orlik-Milnornum}. We give two here, and then describe joint work of Jesse Kass and the second named author enriching the equality between them \cite[$\S6$]{MR3909901}.

When $X$ is the hypersurface $X = \{f=0\} \subset \C^n$, the singular locus is the closed subscheme determined by $f=0$ and $\grad f =0$. Suppose $x \in X$ is an isolated critical point of $f$. Since $\grad f$ has an isolated zero at $x$, we may take the local Brouwer degree $\deg_x \grad f$. The {\em Milnor number} $\mu_x (X)$ is this local (topological) degree $$\mu_x (X) = \deg_x \grad f.$$ 

Another point of view on the Milnor number is as follows. A point $x$ on a complex hypersurface $X$ is called a {\em node} if the completed local ring $\hat{\calO}_{X,x}$ is isomorphic to $$\frac{\C[[x_1,\ldots,x_n]]_x}{x_1^2+\ldots x_n^2 + \text{ higher order terms}} $$ Equivalently, the determinant of the Hessian does not vanish at nodes. Nodes are the simplest singularity, and generically, a singularity will bifurcate into nodes. Milnor shows that the number of these nodes is the Milnor number \cite[p.113]{MR0239612}.

\begin{example}\label{cusp_mu_classical}
The \emph{cusp} is defined by the equation $f=x_2^2-x_1^3=0$ in $\C^2$. It has one isolated singularity at $0$ with Milnor number equal to $$\deg_0((x_1,x_2) \mapsto (-3 x_1^2, 2 x_2)) =\deg_0(x_1 \mapsto -3x_1^2)\deg_0(x_2 \mapsto 2 x_2) = 2*1=2.$$ Consider instead the perturbation $$f_t = x_2^2 - x_1^3 - tx_1$$ and the one-parameter family of hypersurfaces \begin{equation}\label{cusp_family_A1u}f_t(x_1,x_2) = u\end{equation} over $\A^1_u = \Spec \C[u]$. The hypersurface \eqref{cusp_family_A1u} has a singularity if and only if the cubic equation $x_1^2 + t x_1 + u$ has a double root. This happens if and only if the discriminant $-4t^3 - 27u^2$ is $0$. When $t=0$, we see that we have one singular point, which is the cusp we started with. When we fix a particular $t$ with $t \neq 0$, then we have $2$ singular points, both of which are nodes. As $t$ moves away from $0$, the cusp bifurcates into these $2$ nodes, verifying Milnor's equality in this case. See the figure below.

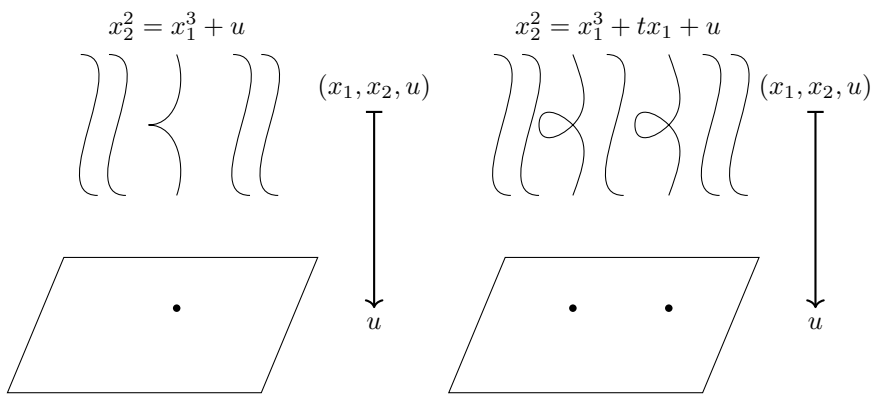
\tikzset{
  pics/curvy/.style = {
    code = {
      \draw (0, 0) to[out=0, in=180] (.3, -2.5);
    }
  },
  pics/nodey/.style = {
    code = {
      \draw (0, 0) to[out=-70, in=45] (0, -1.25)
        to[out=-135, in=-90, looseness=1.5] (-.6, -1.25)
        to[out=  90, in=135, looseness=1.5] (0, -1.25) to[out=-45, in=70] (0, -2.5);
    }
  }
}

\tikzstyle{point} = [circle, fill=black, inner sep=0pt,
minimum size=1mm, anchor=center]

\begin{center}
  \begin{tikzpicture}[x=.75cm, y=.75cm]
    \node at (3, 6.5) {$x_2^2 = x_1^3 + u$};
    \pic at (1.3, 6) {curvy};
    \pic at (1.8, 6) {curvy};
    \pic at (4, 6) {curvy};
    \pic at (4.5, 6) {curvy};
    \draw[shift={(3,6)}]
    (0, 0) to[out=-70, in=0] (-.5, -1.25) to[out=0, in=70] (0, -2.5);
    \node[point] at (3, 1.5) {};
    \draw (0, 0) -- (1, 2.4) -- (5.5, 2.4) -- (4.5, 0) -- cycle;
    \draw[|->, thick] (6.5, 5) node[above] {$(x_1, x_2, u)$} -- (6.5, 1.5) node[below] {$u$};
  \end{tikzpicture}
  \begin{tikzpicture}[x=.75cm, y=.75cm]
    \node at (3, 6.5) {$x_2^2 = x_1^3 + tx_1 + u$};
    \pic at (.8, 6) {curvy};
    \pic at (1.3, 6) {curvy};
    \pic at (2.2, 6) {nodey};
    \pic at (2.8, 6) {curvy};
    \pic at (3.9, 6) {nodey};
    \pic at (4.5, 6) {curvy};
    \pic at (5, 6) {curvy};
    \node[point] at (2.2, 1.5) {};
    \node[point] at (3.9, 1.5) {};
    \draw (0, 0) -- (1, 2.4) -- (5.5, 2.4) -- (4.5, 0) -- cycle;
    \draw[|->, thick] (6.5, 5) node[above] {$(x_1, x_2, u)$} -- (6.5, 1.5) node[below] {$u$};
  \end{tikzpicture}
  \end{center}

\end{example}

\subsection{$\A^1$-Milnor numbers}
\label{subsection: A1 Milnor numbers}

In \cite[$\S6$]{MR3909901} Kass and the second named author define an enriched version of the \emph{Milnor number} of hypersurface singularities, and then use the EKL-form to compute it. See also \cite{pauli2020computing} for computations of $\A^1$-Milnor numbers using Macaulay2. The definition applies to isolated zeros $x$ of $\grad f$, where $f$ is the equation determining the hypersurface. When $X$ is the hypersurface $X = \{f=0\} \subset \PP_k^n$ over a field $k$, the singular locus is the intersection of $X$ and the closed subscheme determined by $Z= \{\grad f =0\} \subseteq \PP^n_k$, and the assumption that $x$ is an isolated zero allows us to take the local $\A^1$-degree $\deg_x^{\A^1}(\grad f)$. Furthermore, the local ring $$\calO_{Z,x} \cong k[x_0,\ldots, x_n]_{x}/(\partial_0 f, \ldots, \partial_n f)$$ is a finite dimensional $k$-algebra with a distinguished presentation, giving an EKL-form computing $\deg_x^{\A^1}(\grad f)$.
\begin{definition}
Let $\{f=0\}=X\subset\A^n$ be a hypersurface with an isolated singularity at a point $x$. We set $\mu_x^{\A^1}(f):=\deg_x^{\A^1}(\grad f)$. 

Over $\C$, the generic singularity has completed local ring $\C[[x_1,\ldots,x_n]]/(x_1^2+\ldots+x_n^2)$, and we called such singularities nodes. Over non-algebraically closed fields, nodes carry interesting arithmetic information. For example, over $\R$, there are three types of nodes in the plane: the split node, defined by $x_1^2 = x_2^2$, the non-split node, given by $x_1^2 = -x_2^2$, and a complex conjugate pair of nodes.

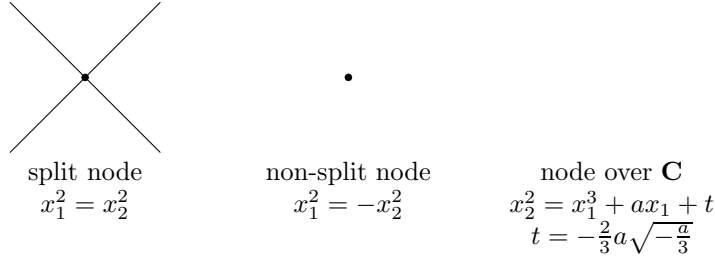
\tikzstyle{point} = [circle, fill=black, inner sep=0pt,
minimum size=1mm, anchor=center]
\begin{center}
  \begin{tikzpicture}
    \node[point] at (0, 0) {};
    \draw (-1, -1) -- (1, 1) (1, -1) -- (-1, 1);
    \node[below, align=center] at (0, -1) {split node\\$x_1^2 = x_2^2$};

    \begin{scope}[xshift=3.5cm]
      \node[point] at (0, 0) {};
      \node[below, align=center] at (0, -1) {non-split node\\$x_1^2 = -x_2^2$};
    \end{scope}

    \begin{scope}[xshift=7cm]
      \node[below, align=center] at (0, -1)
      {node over $\mathbf C$\\
        $x_2^2 = x_1^3 + ax_1 + t$\\
      $t = -\frac23 a \sqrt{-\frac a3}$};
    \end{scope}

  \end{tikzpicture}
  \end{center}

To study nodes, we assume $$\characteristic k \neq 2,$$ and define a {\em node} to be a point on a finite-type $k$-scheme $X$ such that for all the points $\tilde{x}$ of the base change $X_{\kbar}$ of $X$ to the algebraic closure of $k$, the completed local ring $\hat{\calO}_{X_{\kbar},\tilde{x}}$ is isomorphic to $$\kbar[[x_1,\ldots,x_n]]/(x_1^2+x_2^2+\ldots+x_n^2 + \text{ higher order terms} ) $$ See \cite[Expos\'e XV]{MR0354657} for more information. 

\begin{example}
The $\A^1$-Milnor number of a node records information about its field of definition and tangent directions.  

Consider first the node $x=(0,0)$ of the plane curve given by $f(x_1, x_2) = a_1 x_1^2 + a_2 x_2^2 = 0$. Then $\mu_{x}^{\A^1}(f) = \deg^{\A^1}_0  (2 a_1 x_1, 2 a_2 x_2) = \lra{a_1a_2}$. The element $a_1a_2$ in $k^*/(k^*)^2$ has a geometric interpretation: the field of definition of the two lines $x_1 = \sqrt{\frac{-a_2}{a_1}} x_2$ and $x_1 = -\sqrt{\frac{-a_2}{a_1}} x_2$ making up the tangent cone is $k(\sqrt{-a_1a_2})$. A node is called {\em split} if these two lines are defined over $k$ and {\em non-split}  otherwise. More generally, given a rational point $x$ which is a node of a plane curve $\{ f=0\} \subset \PP^2_k$, let $D$ in $k^*/(k^*)^2$ such that the lines of the tangent cone to $f$ at $p$ are defined over $k(\sqrt{D})$. Then $\mu_x^{\A^1}(f) = \lra{-D}$. 

The field of definition of any node is separable \cite[Expos\'e XV, Th\'eoreme 1.2.6]{MR0354657}, so given a node $x$ on a plane curve $\{ f=0\} \subset \PP^2_k$ we can reduce to the case of a rational node using Example \ref{example: local A1 degree separable field extension}. Namely, we have a tower of field extensions $k \subseteq k(x) \subseteq k(x)[\sqrt{D}]$ where $D$ in $k(x)^*/(k(x)^*)^2$ is chosen so that $ k(x)[\sqrt{D}]$ is the field of definition of the lines in the tangent cone. Then $$\mu_x^{\A^1}(f) = \Tr_{k(x)/k} \lra{-D}.$$ 
 
In higher dimensions, we have for $f(x_0,\ldots, x_n) = a_1 x_1^2 + a_2 x_2^2 + \ldots + a_n x_n^2 + \text{ higher order terms}$ and $x=[1,0, \ldots, 0]$ that the $\A^1$-Milnor number is given by $$\mu_x^{\A^1}(f) = \langle 2^n\prod_{i=1}^n a_i \rangle,$$ and this gives the general case as we may similarly assume the node is at a rational point using $\mu_x^{\A^1}(f) = \Tr_{k(x)/k} \mu_x^{\A^1}(f \otimes k(x)).$
\end{example}

\begin{definition}\label{df:type}
For a node $x$ on a hypersurface $\{ f = 0 \}$ in affine or projective space, the {\em type} of $x$ is defined to be  $$\type(x) :=  \mu_x^{\A^1}(f \otimes k(x)).$$ 
\end{definition}

We also write $\type(x,f) = \type(x)$ to emphasize the dependence on $f$ when the dimension $n$ of the ambient affine or projective space is odd. When $n$ is even, $\type(x)$ is an invariant of the singularity, meaning it only depends on the completed local ring of $X = \{f=0\}$ and $x$, and notably not on the choice of $f$ itself \cite[Lemma 39]{MR3909901}.  When $n$ is odd, the type of $f$ will scale by $\langle a \rangle$ when $f$ is replaced by $af$.

Note that for a plane curve $\{ f = 0 \}$, the type of a node records the field of definition of the two tangent directions at the node (i.e. the two lines making up the tangent cone), and more generally, the type records information about the tangent cone to $\{ f = 0 \}$ at $x$. 

In general the $\A^1$-Milnor number of $f$ is an invariant of $f$ and the singularity $x$. Kass and the second named author show that the sum the $\A^1$-Milnor numbers of the singularities of $\{ f = 0\}$ is equal to a weighted count of nodes of hypersurfaces in a perturbed family.  This is written for the case where $n$ is even, but that is to be able to apply \cite[Lemma 39]{MR3909901}. It is not necessary for the proof: recording the information of $f$, an analogous result holds. 

More precisely, it is shown that for general $(a_1,\dots,a_n)\in \A^n(k)$
the family \[f(x_1,\dots,x_n)+a_1x_1+\dots+a_nx_n=t\]
over the affine $t$-line has only nodal singularities \cite[Lemma 43]{MR3909901} and $$\sum_{x \text{ singularity of } \{f=0\}}\mu^{\A^1}_x(f)$$ is equal to the sum $$\sum_{x \text{ node of } \{f(x)+ax=t\}} \Tr_{k(x)/k} \type(x) $$ of $\Tr_{k(x)/k} \type(x)$, where $x$ runs over the nodes of hypersurfaces in the $t$-family $f(x_1,\dots,x_n)+a_1x_1+\dots+a_nx_n-t)$ for fixed generic $(a_1,\ldots,a_n)$ in $k^n$.

\begin{theorem}\label{thm:Cor45KW19}
\cite[Corollary 45]{MR3909901} Let $f \in k[x_1,\ldots, x_n]$ be such that $\grad f$ is finite and separable. Then for $(a_1,\ldots,a_n) \in \A^n_k(k)$ a general $k$-point, the family $$\A^n_k \to \A^1_k $$ \begin{equation}\label{family}x \mapsto f(x) - a_1 x -\ldots -a_n x_n \end{equation} has only nodal fibers. Suppose that the residue field of every zero of $\grad(f)$ is separable over $k$. Then we have an equality $$\sum_{x \text{ singularity of } \{f=0\}}\mu^{\A^1}_x(f) = \sum_{x \text{ node of \eqref{family}}} \Tr_{k(x)/k} \type(x,f) .$$ 

\end{theorem}

\begin{proof}
The proof of \cite[Corollary 45]{MR3909901} in \cite{MR3909901} gives a statement with the additional hypotheses that $n$ is even and that every zero of $\grad(f)$ either has residue field $k$ or is in the \'etale locus of $\grad f$. The first hypothesis is removed by including the information of $f$ into $\type(x,f)$. The second hypothesis was present to ensure with the technology available at the time that the $\A^1$-local degree agrees with a bilinear form constructed in \cite[Satz 3.3]{MR393056}, which will be described here in Section \ref{subsection: dynamic interpretation}. It is weakened to the hypothesis that the zeros of $\grad(f)$ have residue field which is separable over $k$ by \cite[Theorem 1.3]{brazelton2019trace}, \cite[main theorem]{MR3909901} and \cite[Proposition 32]{2017arXiv170801175K}.
\end{proof}

\end{definition}
\begin{example}[Cusp continued]
In Example \ref{cusp_mu_classical}, we looked at the classical Milnor number of the cusp defined by $f=x_2^2-x_1^3$, and its bifurcation into nodes. We now enrich this example using Theorem \ref{thm:Cor45KW19}.  The $\A^1$-Milnor number of the cusp is 
\[\mu^{\A^1}(f)=\deg^{\A^1}_0\grad f=\deg^{\A^1}_0(3x_1^2,2x_2)=\langle1\rangle+\langle-1\rangle\in \GW(k).\]  (To see this, one can express $\deg^{\A^1}_0(3x_1^2,2x_2)$ as the product $$\deg^{\A^1}_0(3x_1^2,2x_2) = \deg^{\A^1}_0 (3 x_1^2)\deg^{\A^1}_0 (2 x_2) = \langle 3 \rangle \deg^{\A^1}_0 (x_1^2)\langle 2 \rangle ,$$ and the $\A^1$-degree $\deg^{\A^1}_0 (x_1^2)$ was computed to be $\langle1\rangle+\langle-1\rangle$ in Example \ref{EKLformz2}.)  As in Example \ref{cusp_mu_classical}, the cusp bifurcates into $2$ nodes. These nodes are either a pair of conjugate nodes defined over a separable degree $2$ extension of $k$, or $2$ rational nodes. For each of these nodes, the lines in the tangent cone have some fields of definition. Theorem \ref{thm:Cor45KW19} gives restrictions on what field extensions and tangent directions are possible, or in other words, Theorem \ref{thm:Cor45KW19} gives restrictions on the types of these nodes. For example, suppose the field $k$ is the finite field $\FF_5$ with $5$ elements. Then $\langle1\rangle+\langle-1\rangle$ has trivial discriminant. So it is not possible for any choice of perturbation for the cusp to bifurcate into $2$ rational nodes with one split and one non-split. Similarly, it is not possible for the cusp to bifurcate into a pair of conjugate nodes over the unique degree $2$ extension which are split, because $\Tr_{\FF_{5^2}/\FF_5} \langle -1 \rangle$ has nontrivial discriminant.

However, if instead $k=\FF_7$, then the cusp can not bifurcate into $2$ split rational nodes, or $2$ non-split rational nodes. The cusp over $\FF_7$ can also not bifurcate into pair of conjugate nodes over the unique degree $2$ extension which are split, because $\Tr_{\FF_{7^2}/\FF_7} \langle -1 \rangle$ has trivial discriminant. 
\end{example}

We want to give a different \emph{dynamic} interpretation of the $\A^1$-Milnor number using the \emph{dynamic} local degree used in \cite{pauli2020quadratic} to compute the local contributions of the 2875 distinguished lines on the Fermat quintic threefold. We also remove the sum on the left hand side, replacing it with an equation for $\mu^{\A^1}_x(f)$ as a sum of the nodes the $x$ bifurcates into. In practice, this happens with \cite[Corollary 45]{MR3909901} as well, for example in the cases where $x$ is the only singularity of $\{ f = 0\}$ or when the other singularities are nodes which remain nodes and make the same contribution to each side. However, it is more aesthetically pleasing to identify the nodes that the singularity bifurcates into and then have an equality between traces of types of these nodes and the $\A^1$-Milnor number of the singularity. This is what we do in Theorems \ref{A1-Milnor_number_sum_A1-Milnor_numbers_bifurcations} and \ref{A1-Milnor_number_sum_node-types}.

\subsection{A dynamic interpretation of the $\A^1$-Milnor number}
\label{subsection: dynamic interpretation}

Let $p$ be a singular point of the hypersurface $X_0 = \{ f = 0\} \hookrightarrow \A^n_k$, where $f$ is in $k[x_1,\ldots,x_n]$ and $k$ is a field. (We could also take $X_0 \hookrightarrow \PP^n_k$ and $f$ homogenous in $k[x_0,\ldots,x_n]$.) Assume that $\grad f$ has an isolated zero at $p$, allowing the $\A^1$-Milnor number $\mu^{\A^1}(f,p)$ of $f$ at $p$ to be defined, as discussed above. We use S. Pauli's enrichment of dynamic intersection numbers to allow non-linear deformations of $f$ in Theorem~\ref{thm:Cor45KW19}, and replace the sum by the $\A^1$-Milnor number itself: We show that under a generic deformation of $f$ over $k$, the singularity $p$ bifurcates into nodes, and letting $\operatorname{Nodes}$ denote the set of these nodes, we have that the $\A^1$-Milnor number at $p$ is the sum $$\mu^{\A^1}(f,p) = \sum_{x \in \operatorname{Nodes}(p)} \Tr_{k(x)/k}\type(x).$$ As above, $\type(x)$ is the type of Definition \ref{df:type}, and records information about the tangent cone at $p$. 

For $g \in k[x_1,\ldots,x_n][[t]]$, consider $f_t = f + t g$ in $k[x_1,\ldots,x_n][[t]]$, defining a deformation $$X = \{ f + t g = u \} \hookrightarrow  \A^n_{k[u][[t]]}$$ of $X$. Let $Y = \{ \grad (f + t g) = 0 \} \hookrightarrow  \A^n_{k[[t]]}$, where $\grad$ denotes the $n$-tuple of partial derivatives with respect to the variables $x_i$ for $i=1,\ldots,n$.  

We will let a $0$-subscript denote the special fiber of a scheme over $\Spec k[[t]]$, e.g., $Y_0 = \Spec k \times_{\Spec k[[t]]} Y$. Then $Y_0$ corresponds to the singularities of the family of varieties $\{ f = u \}$ parametrized by $\Spec k[u]$ and the generic fiber $$Y_{\gen} = \Spec k((t)) \times_{\Spec k[[t]]} Y$$ of $Y$ corresponds to the singularities of the family of varieties $ \{ f + t g = u \}$ parametrized by $\Spec k((t))[u]$.

By \cite[Lemma 10.152.3. (12) Tag 04GE]{stacks-project}, $Y = Y^f \coprod Y^{\geq 1}$, where $Y^f \to \Spec k[[t]]$ is finite and $Y^{\geq 1}$ has all components of its special fiber of dimension $\geq 1$, i.e. $Y^{\geq 1}_0 $ is a union of positive dimensional $k$-varieties. Let $Y^p$ be the union of the irreducible components of $Y$ containing $p$. Since $\grad f$ has an isolated $0$ at $p$, the ring $\calO_{Y_0,p}$ is a finite $k$-module, and it follows that $Y^p$ is a closed subscheme of $Y^f$. Thus $Y^p \to \Spec k[[t]]$ is finite.

\begin{lemma}\label{calOY^p_0_local}
$p$ is the only point of $Y^p_0$, and $\Gamma(\calO_{Y^p})$ is a local ring.
\end{lemma}

\begin{proof}
Consider the pullback diagram $$ \xymatrix{ Y^p_{\gen} \ar[d] \ar[r]^{\eta'} & Y^p \ar[d] & \ar[l]^{s'} \ar[d] Y^p_0 \\ \Spec k((t)) \ar[r]^{\eta} & \Spec k[[t]] & \ar[l]^s \Spec k[[t]]/\langle t \rangle } $$ A point $x$ of $Y^p_{\gen}$ has a field of definition $L := k(x)$ which is a finite extension of $k((t))$, and therefore a complete valued field. The integral closure $R$ of $k[[t]]$ in $L$ is the ring of integers of $L$ and is finite over $k[[t]]$ by \cite[Proposition 6.4.1/2, Chapter 6, p. 250]{MR746961}. Applying the valuative criteria of properness, we have a unique diagonal arrow in the commutative diagram $$
\xymatrix{\Spec L \ar[r] \ar[d]& Y^p \ar[d]\\
\Spec R \ar[r] \ar@{.>}[ur]& \Spec k[[t]]}
$$ which is moreover a finite map because $R$ is finite over $k[[t]]$ and $Y^p \to \Spec k[[t]]$ is separated. The image of $\Spec R$ in $Y^p$ is therefore a closed $1$-dimensional subscheme of $Y^p$, whence a component. Therefore it contains $p$. However $\Spec R$ has a unique point in the special fiber \cite[Theorem 3.2.4/2 Chapter 3 p. 139]{MR746961}. It follows that $p$ is the only point of $Y^p_0$. It follows from this that $\Gamma(\calO_{Y^p})$ is a local ring.

\end{proof}

The points of $Y^p$ are the singular point $p$ and the singularities $p$ bifurcates into. The latter are the singularities in the $\Spec k((t))[u]$-family $ \{ f + t g = u \}$ and are in one to one correspondence with points of $Y^p_{\gen}$.

Let $f_i$ and $g_i$ denote the partial derivatives $f_i : = \partial_{x_i} f$ and $g_i : = \partial_{x_i} g$, respectively. Since $Y^p$ is an open subset of $Y \cong \frac{k[x_1,\ldots,x_n][[t]]}{(f_1 + t g_1, \ldots, f_n + t g_n)},$ there is a multiplicatively closed subset $S \subset k[x_1,\ldots,x_n][[t]]$ such that $$Y^p \cong \frac{S^{-1}(k[x_1,\ldots,x_n][[t]])}{(f_1 + t g_1, \ldots, f_n + t g_n) }.$$ Let $\mathfrak{m} \subset k[x_1,\ldots,x_n][[t]]$ denote the maximal ideal containing $t$ corresponding to the point $p$. Since $\grad f$ has an isolated $0$ at $p$, the ring $\calO_{Y_0,p}$ is a finite $k$-module, and $\grad f$ determines a regular sequence in the local ring $k[x_1, \ldots, x_n]_{m_p}$, where $m_p=\mathfrak{m} \cap k[x_1, \ldots, x_n]$ denotes the prime ideal corresponding to $p$. 

\begin{proposition}
$Y^p \to \Spec k[[t]]$ is finite, and flat. Furthermore, $$f_1 + t g_1, \ldots, f_n + t g_n$$ is a regular sequence in the localization $S^{-1}(k[x_1,\ldots,x_n][[t]])_{\mathfrak{m}}$
\end{proposition}

\begin{proof}
We have already seen that $Y^p$ is finite over $\Spec k[[t]]$. Since $$S^{-1}(k[x_1,\ldots,x_n][[t]])$$ is regular of dimension $n+1$, it is Cohen-Macaulay. Moreover the quotient $S^{-1}(k[x_1,\ldots,x_n][[t]])/(f_1 + t g_1, \ldots, f_n + t g_n,t) \cong \calO_{Y^0,p}$ is a finite, local ring of dimension $0$ by the assumption that $p$ is an isolated zero of $\grad f$. It follows from \cite[Lemma 10.103.2 TAG 00N7]{stacks-project} that $f_1 + t g_1, \ldots, f_n + t g_n,t$ is a regular sequence in $S^{-1}(k[x_1,\ldots,x_n][[t]])_{\mathfrak{m}}$ and the quotient $$S^{-1}(k[x_1,\ldots,x_n][[t]])_{\mathfrak{m}}/(f_1 + t g_1, \ldots, f_n + t g_n)$$ is Cohen-Macaulay of dimension $1$. Since  $\Gamma(\calO_{Y_p})$ is a local ring, we may remove the previous localization at $\mathfrak{m}$ giving the statement that $S^{-1}(k[x_1,\ldots,x_n][[t]])/(f_1 + t g_1, \ldots, f_n + t g_n)$ is Cohen-Macaulay of dimension $1$.  It follows from \cite[Theorem 23.1 p. 179]{Matsumura_CRT} that $Y^p \to \Spec k[[t]]$ is also flat, proving the proposition. 
 
\end{proof}

Since $\calO_{Y^p}$ is flat over $k[[t]]$ it is a locally free, and even free $k[[t]]$-module. The presentation $\calO_{Y^p} \cong \frac{S^{-1}(k[x_1,\ldots,x_n][[t]])}{(f_1 + t g_1, \ldots, f_n + t g_n) }$ moreover determines a $k[[t]]$-bilinear form over $\calO_{Y^p}$ in the following manner. 

The regular sequence $f_1 + t g_1, \ldots, f_n + t g_n$ determines a distinguished isomorphism $$\chi(\Delta): \Hom_{k[[t]]}(\calO_{Y^p}, k[[t]])  \stackrel{\cong}{\to} \calO_{Y^p}$$ following work of Scheja and Storch \cite{MR393056}, giving a version of the Eisenbud--Levine/Khimshiashvili form which works in families. Namely, we can choose $a_{ij}$ in $k[x_1,\ldots,x_n][[t]] \otimes_{k[[t]]} k[x_1,\ldots,x_n][[t]]$ such that $$ (f_i + t g_i) \otimes 1 - 1 \otimes (f_i + t g_i) = \sum_j a_{ij}(x_j \otimes 1 - 1 \otimes x_j).$$ Let $\Delta$ denote the image of $\det (a_{ij})$ in $\calO_{Y^p} \otimes \calO_{Y^p}$. It is shown  \cite[Satz 3.1]{MR393056} that $\det (a_{ij})$ is independent of the choice of $a_{ij}$. Let $$\chi: \calO_{Y^p} \otimes \calO_{Y^p} \to \Hom_{k[[t]]}(\Hom_{k[[t]]}(\calO_{Y^p}, k[[t]]), \calO_{Y^p})$$ denote the map $$b \otimes c \mapsto (\phi \mapsto \phi(b) c)$$ Scheja and Storch show \cite[Satz 3.3]{MR393056} that $\chi(\Delta)$ is an isomorphism. 

Let $\ev_1: \Hom_{k[[t]]}(\calO_{Y^p}, k[[t]]) \to k[[t]]$ denote the evaluation at $1 \in \calO_{Y^p}$, sending $\eta$ in $\Hom_{k[[t]]}(\calO_{Y^p}, k[[t]])$ to $\eta(1)$. $\ev_1$ corresponds to the trace \cite[p.7 (b)3 Ideal theorem]{MR0222093}. Thus by Grothendieck--Serre duality  \cite[p 7 b) c) Ideal theorem]{MR0222093}, the composition \begin{equation}\label{eq:mupf_tg}\calO_{Y^p} \times \calO_{Y^p} \to  \calO_{Y^p} \stackrel{\chi(\Delta)^{-1}}{\to}  \Hom_{k[[t]]}(\calO_{Y^p}, k[[t]]) \stackrel{\ev_1}{\to} k[[t]] \end{equation} of multiplication with $\chi(\Delta)^{-1}$ and $\ev_1$ is non degenerate.

\begin{definition}
Let $\mu^{\A^1}_p(f+tg)$ be the element of $\GW(k[[t]])$ corresponding to the pairing \eqref{eq:mupf_tg}.
\end{definition}

We have maps $\GW(k[[t]]) \to \GW(k)$ and $\GW(k[[t]]) \to \GW(k((t)))$ associated to the ring maps $k[[t]] \to k$ and $k[[t]] \to k((t))$. By construction, the image of $\mu^{\A^1}_p(f+tg)$ in $\GW(k)$ is $\mu^{\A^1}_p(f)$ and the image in $\GW(k((t)))$ is the sum over the points of the generic fiber $ x \in Y^p_{\gen}$ of $\mu^{\A^1}_x (f+tg - u(x))$. Since $u(x)$ does not effect the pairing on $k((t))[x_1,\ldots, x_n]/ (f_1 + t g_1, \ldots, f_n + t g_n)$, it is natural to let $\mu^{\A^1}_x(f+ tg) = \mu^{\A^1}_x (f+tg - u(x))$.

\begin{example}\label{example:cusp_continuted_with_node_equations}[Cusp continued]
Recall that for the cusp equation $f=x_2^2-x_1^3$ the $\A^1$-Milnor number $\mu_0^{\A^1}(f)$ is equal to the hyperbolic form $\langle1\rangle+\langle-1\rangle\in \GW(k)$. So we expect the singularity of the cusp to bifurcate into two nodes such that the sum of the types of these nodes is the hyperbolic form.

Let $g=3x_1+2x_2+2x_1^3-tx_1^3$. Then $f+tg$ has partial derivatives
\[\partial_{x_1}(f+tg)=-3x_1^2+3t+6tx_1^2-3t^2x_1^2\]
and 
\[\partial_{x_2}(f+tg)=2x_2+2t.\]
Setting both partial derivatives equal to zero, we get that f+tg has two critical points, namely
\[x_1=\frac{\sqrt{t}}{1-t}\text{, }x_2=-t\]
and
\[x_1=-\frac{\sqrt{t}}{1-t}\text{, }x_2=-t\]
both defined over $k((t^{1/2}))$. The sum of the $\A^1$-Milnor numbers at these nodes is
\begin{align*}\Tr_{k((t^{1/2}))/k((t))}(\mu^{\A^1}_{(\frac{\sqrt{t}}{1-t},-t)}(\grad(f+tg)))&=\Tr_{k((t^{1/2}))/k((t))}(\langle12\sqrt{t}(1-t)\rangle)\\
&=\langle1\rangle+\langle-1\rangle\in \GW(k((t))).\end{align*}

\end{example}

We have that the $\A^1$-Milnor number at $p$ is the sum of the $\A^1$-Milnor numbers of the singularities $p$ bifurcates into.

\begin{theorem}\label{A1-Milnor_number_sum_A1-Milnor_numbers_bifurcations}
Let $k$ be a field and let $X = \{f=0\}$ determine a hypersurface in $\A^n_k$. Let $p$ be a singularity of $X$ which is an isolated zero of $\grad f$.\footnote{The condition that $p$ is an isolated zero of $\grad f$ is implied by $p$ being an isolated singularity of $X$ if the characteristic of $k$ is $0$.} Then for any $g$ in $k[x_1, \ldots, x_n][[t]]$, the $\A^1$-Milnor number $\mu^{\A^1}_p(f)$ of $f$ at $p$ equals the sum of the $\A^1$-Milnor numbers of the singularities of the deformation $\{ f+ tg = u\}$ that $p$ bifurcates into: $$\mu^{\A^1}_p(f) = \sum_{x \in Y^p_{\gen}} \mu^{\A^1}_x (f+ tg).$$ Here, $\GW(k)$ is viewed as a subring of $\GW(k((t)))$ via the canonical injection. In particular, the right hand side is necessarily in $\GW(k)$.
\end{theorem}

\begin{remark}
Recall that $Y^p$ was defined above to be the union of the components of $Y = \{ \grad (f + t g) = 0 \} \hookrightarrow  \A^n_{k[[t]]}$ containing $p$, and $Y^p_{\gen}$ denotes its generic fiber. Its points are the singularities of the $u$-family of deformations $\{ f+ tg = u\}$ that $p$ bifurcates into.
\end{remark}

\begin{proof}
We showed above that $\mu^{\A^1}_p (f)$ is the image under $\GW(k[[t]]) \to \GW(k)$ of a well-defined $\mu^{\A^1}_p(f+tg)$ in $\GW(k[[t]])$. The map $\GW(k[[t]]) \to \GW(k)$ is an isomorphism with inverse given by the map corresponding to the inclusion of rings $k \subset k[[t]]$. The sum $\sum_{x \in Y^p_{\gen}} \mu^{\A^1}_x (f+ tg)$ is the image of  $\mu^{\A^1}_p(f+tg)$ under $\GW(k[[t]]) \to \GW(k((t)))$, whence it equals $\mu^{\A^1}_p(f)$ as claimed.
\end{proof}

We now specialize Theorem \ref{A1-Milnor_number_sum_A1-Milnor_numbers_bifurcations} to the case where $p$ bifurcates into nodes, where it becomes the statement that the $\A^1$-Milnor number of $p$ is the sum of the types of these nodes, enriching the result described at the beginning of Section \ref{subsection: C Milnor numbers}.

The condition that $p$ bifurcates into nodes is equivalent to the statement that the Hessian (determinant) of $f+tg$ is non-zero at all the singularities $p$ bifurcates into. Since the Hessian determinant is the Jacobian element of $\grad(f + tg)$, this is equivalent to the statement that $Y^p_{\gen} \to \Spec k((t))$ is \'etale.

We give some criteria for this to happen.

\begin{proposition}\label{Xeta_etale_g_after_linear_mod_under_separability}
For $h$ in $k[x_1,\ldots, x_n][[t]]$ such that $\grad(f + th): \A^n_{k((t))} \to \A^n_{k((t))}$ is a finite, separable map, there exist infinitely many $(a_1, \ldots, a_n)$ with $a_i$ in $k[[t]]$ for $i=1,\ldots,n$ such that $Y^p_{\gen} \to \Spec k((t))$ is \'etale for $g = h - \sum_{i=1}^n a_i x_i$.
\end{proposition}

The assumption that $\grad(f + th)$ is separable means that the associated extension of function fields is a separable extension and in particular, this is automatic in characteristic $0$.  

\begin{proof}
Since $\grad(f + th): \A^n_{k((t))} \to \A^n_{k((t))}$ is a separable map, it is generically \'etale. Thus there is a non-empty open subset of points at which $\grad(f+th)$ is \'etale. The image of the complement is closed because $\grad(f + th)$ is finite. Thus there is a non-empty open subset $U \subseteq \A^n_{k((t))}$ such that $\grad(f + th)$ is \'etale on points of $\grad(f+th)^{-1}(U)$. 

We claim that $U$ contains infinitely many points of the form $(t a_1, \ldots, t a_n)$ with $a_i$ in $k[[t]]$. The complement of $U$ is a proper closed subset of $\A^n_{k((t))}$. It is therefore contained in the zero locus of some polynomial $P$. Fixing $n-1$ of the variables to be values of the form $t a_i$ with $a_i$ in $k[[t]]$ such that the resulting polynomial in the last variable is not the zero polynomial (which is possible by induction) results in finitely many excluded values for the last variable.

For any point $(t a_1, \ldots, t a_n)$ with $a_i$ in $k[[t]]$, we claim the deformation $g = h - \sum_{i=1}^n a_i x_i$ has the desired property. For points of $\A^n_{k((t))}$ where $\grad (f+ t g) = 0$, we have that $\grad(f+ th) = (t a_1, \ldots t a_n)$. By choice of the $a_i$, we have that the Jacobian determinant of $\grad(f + th)$ is non-zero. This Jacobian determinant is also called the Hessian (determinant) of $f+t h$, which equals the Hessian of $f+ tg$. Thus the Hessian of $(f+ t g)$ is non-zero at the zero locus of $\grad (f+ t g) = 0$ in $\A^n_{k((t))}$. Thus $Y_{\gen} \to \Spec k((t))$ and in particular $Y^p_{\gen} \to \Spec k((t))$ is \'etale.
\end{proof}

We wish to make a precise statement of the form that for a generic deformation $g$, the singularity $p$ bifurcates into nodes. One option is the following Proposition \ref{Xeta_etale_generic_g}. The hypothesis on the behavior of $f$ at infinity should be irrelevant under an appropriate reformulation, but we keep it here for present lack of a better option. 

\begin{proposition}\label{Xeta_etale_generic_g}
Let $d$ be the degree of $f$ in $k[x_1, \ldots, x_n]$, and let $F$ in $k[x_0,\ldots, x_n]$ denote the degree $d$ homogenization of $f$. Suppose that $\partial_{x_i} F= 0$ for $i>0$ has no solutions in $\{ x_0 = 0\} \hookrightarrow \PP^n_k$ and that $\grad f: \A^n_k \to \A^n_k$ is finite and separable. Then a generic polynomial $g \in k[x_1, \ldots, x_n][[t]]$ of degree $< d$ has the property that $Y^p_{\gen} \to \Spec k((t))$ is \'etale for the deformation $f+ tg$. 
\end{proposition}

\begin{proof}

The space of polynomials $g \in k[x_1, \ldots, x_n][[t]]$ of degree $< d$ is an affine space $\A^N_{k[[t]]}$ for some $N$. Let $G$ denote the degree $d$ homogenization of $g$. The homogenization $F+tG$ of $f+tg$ has no solutions to $\partial_{x_i} (F+ tG)=0$ at points of  $\{x_0=0\}$ since for $x_0 = 0$, we have $\partial_{x_i} (F+t G) = \partial_{x_i} F$ for $i>0$ because the degree of $g$ is less than $d$. For notational simplicity, let $\grad(F+ tG) = (\partial_{x_1} (F+tG), \ldots, \partial_{x_n} (F+tG))$, so we have that $\{ x_0 = 0, \grad(F+tG)=0\}$ is empty.

Consider the projection $\pi_1: \A^N_{k[[t]]} \times \PP^n_{k[[t]]} \to \A^N_{k[[t]]}$. Let $X \hookrightarrow \A^N_{k[[t]]} \times \PP^n_{k[[t]]} $ be the closed subscheme determined by $X = \{(g,x): \grad(F+tG)(x) =0\}$. Let $\Hess(F+tG) = \det (\frac{\partial^2 (F+tG)}{\partial_{x_i}\partial_{x_j}})_{i,j=1}^n$ denote the Hessian (determinant), and let  $Y \hookrightarrow \A^N_{k[[t]]} \times \PP^n_{k[[t]]} $ be the closed subscheme determined by $Y = \{(g,x): \Hess(F+tG)(x) =0\}$.

Since $\pi_1$ is proper, $\pi_1(X \cap Y)$ is a closed subset of $\A^N_{k[[t]]}$ and it suffices to show that this closed subset is not the entirety of $\A^N_{k[[t]]}$. This follows by Proposition \ref{Xeta_etale_g_after_linear_mod_under_separability} applied in the case that $h=0$ is the zero polynomial. 
\end{proof}

When $Y^p_{\gen} \to \Spec k((t))$ is \'etale, its (finitely many) points correspond to nodes on hypersurfaces $\{ f+ tg = u\} \hookrightarrow \A^n_{k((t))}$. These nodes extend to integral points with special fiber $p$ (see the proof of Lemma \ref{calOY^p_0_local}), and all the singularities in the family of hypersurfaces $\{ f+ tg = u\} \hookrightarrow \A^n_{k((t))}$ specializing to $p$ correspond to points of $Y^p_{\gen}$. In other words, the singularity $p$ bifurcates into a set of nodes, and these nodes are the points of $Y^p_{\gen}$. We now denote the set of points of $Y^p_{\gen}$ by $\Nodes(p)$.

When $k$ is {\bf characteristic $0$}, we can find equations for these nodes, as well as their extensions to integral points containing $p$ in the special fiber in the following manner. The assumption on the characteristic implies that the algebraic closure of of $ k((t))$ is $\cup_{k \subseteq L, n} L((t^{1/n}))$ \cite[IV Section 2 Proposition 8]{Serre-Local_fields}.\footnote{The reference proves the claim for $k$ algebraically closed. The stated result follows by showing that the coefficients of an algebraic power series lie in a finite extension of $k$. Moreover, by \cite{Kedlaya_alg_closure_power_series_positive_characteristic} \cite{Kedlaya_alg_generalized_pwr_series} a perfect extension of a tamely ramified extension of $k((t))$ lies in $\cup_{k \subseteq L, n} L((t^{1/n}))$, even without the assumption on the characteristic.} A point of $Y^p_{\gen}$ at which $Y^p_{\gen} \to \Spec k((t))$ is \'etale therefore determines a commutative diagram $$\xymatrix{\Spec L ((t^{1/n})) \ar[r] \ar[d] & Y^{p} \ar[d]\\ \Spec L [[t^{1/n}]] \ar[r] \ar@{.>}[ur]&\Spec k[[t]]  }.$$ By the valuative criteria of properness, we have the dotted arrow, whence $$x : \Spec L [[t^{1/n}]] \to  Y^{p}.$$ Therefore our point corresponds to a $n$-tuple of power series in $t^{1/n}$ with coefficients in $L$. See Example \ref{example:cusp_continuted_with_node_equations} for equations in $k((t^{1/2}))$ for the two nodes degenerating to the cusp. 

\begin{theorem}\label{A1-Milnor_number_sum_node-types}
Let $k$ be a field and let $X = \{f=0\}$ determine a hypersurface in $\A^n_k$. Let $p$ be a singularity of $X$ which is an isolated zero of $\grad f$. Then for any $g$ in $k[x_1, \ldots, x_n][[t]]$ such that $Y^p_{\gen} \to \Spec k((t))$ is \'etale, the $\A^1$-Milnor number $\mu^{\A^1}_p(f)$ of $f$ at $p$ equals the sum of the transfers of the types of the nodes that $p$ bifurcates into: $$\mu^{\A^1}_p(f) = \sum_{\Nodes p} \Tr_{k(x)/k((t))} \type(x).$$ Here, $\GW(k)$ is viewed as a subring of $\GW(k((t)))$ via the canonical injection. In particular, the right hand side is necessarily in $\GW(k)$.
\end{theorem}

\begin{remark}
Recall that $Y^p$ was defined above to be the union of the components  $Y = \{ \grad (f + t g) = 0 \} \hookrightarrow  \A^n_{k[[t]]}$ containing $p$, and $Y^p_{\gen}$ denotes its generic fiber. Its points are the singularities of the $u$-family of deformations $\{ f+ tg = u\}$ that $p$ bifurcates into, and the assumption that  $Y^p_{\gen} \to \Spec k((t))$ be \'etale is equivalent to the statement that these singularities are all nodes. See Propositions \ref{Xeta_etale_g_after_linear_mod_under_separability} and \ref{Xeta_etale_generic_g} for conditions under which this occurs. 
\end{remark}

\begin{proof}
All the points of $Y^p_{\gen}$ are nodes because $Y^p_{\gen} \to \Spec k((t))$ is \'etale. By Theorem \ref{A1-Milnor_number_sum_A1-Milnor_numbers_bifurcations}, it thus suffices to show that $\Tr_{k(x)/k((t))} \type(x) = \mu^{\A^1}_x (f+ tg)$. This follows by the separability of the field extension $k((t)) \subseteq k(x)$ \cite[Expos\'e XV, Th\'eoreme 1.2.6]{MR0354657} and Example \ref{example: local A1 degree separable field extension}.
\end{proof}

\section{Enriched counts using Euler numbers}
\label{section: Examples of enriched counts}

\subsection{Enriched Euler number}
\label{subsection: Enriched Euler number}

A vector bundle $V$ on a smooth $k$-scheme $X$ is said to be {\em relatively oriented} by the data of a line bundle $L$ on $X$ and an isomorphism $L^{\otimes 2} \cong \Hom(\det TX, \det V).$

Kass and the second named author define an enriched Euler number of a relatively oriented vector bundle of rank $r$ on a smooth, proper $r$-dimensional scheme. This Euler number is an element of $\GW(k)$ and equals the sum of local $\A^1$-degrees at the isolated zeros of a general section \cite{2017arXiv170801175K}. It can be shown \cite{bachmann2020a1euler} to equal a pushfoward in oriented Chow groups of the Euler class of Barge and Morel \cite{MR1753295} \cite{MR2542148}. This was also studied by M. Levine in \cite{levine2017enumerative} where particular attention is given to the $\GW(k)$-valued Euler characteristic which is the Euler number of the tangent bundle.  

\subsubsection{Lines on a smooth cubic surface}
As an application Kass and the second named author get an enriched count of lines on a cubic surface as the Euler number of the vector bundle $\Sym^3\mathcal{S}^*\rightarrow \Gr(2,4)$.

Let $X\subset\PP^3_k$ be a smooth cubic surface. It is a classical result that $X_{\bar{k}}$ contains 27 lines.

\begin{definition}
Let $l$ be a line on $X$ defined over $k(l)$. Then the Gauss map sending $p\in l$ to its tangent space $T_pX$ in $X$ is a degree 2 map 
\[l\cong \PP^1\rightarrow \PP^1=\text{ lines in }\PP^3\text{ containing }l\]
and the non-trivial element of its Galois group is an involution of the line $l$. 
The fixed points of this involution are defined over $k(l)[\sqrt{D}]$ for some $D\in k(l)^{\times}/(k(l)^{\times})^2$. 
Define the \emph{type} of $l$ to be
\[\operatorname{Type}(l)=\langle D\rangle \in\GW(k(l)).\]
\end{definition}

\begin{theorem}[Kass-Wickelgren]
Assume $\operatorname{char} k\neq 2$ and $X$ a smooth cubic surface. Then 
\[\sum_{l \text{ line on }X}\Tr_{k(l)/k}\operatorname{Type}(l)=15\langle 1\rangle +12\langle -1\rangle \in \GW(k).\]
\end{theorem}

\subsubsection{More enriched Euler numbers}
The enriched Euler number has been used to obtain several more enrichted counts:
In \cite{pauli2020quadratic} the first named author defines the \emph{type} of a line on a quintic threefold and uses a dynamic intersection approach to compute an enriched count of lines on a quintic threefold. We saw a similar dynamic approach when we discussed enriched Milnor numbers. 

The Euler numbers corresponding to counts of lines on generic hypersurfaces of degree $2n-1$ in $\PP^{n+1}$ was computed in \cite{LEVINE_2019}. The Euler numbers corresponding to counts of $d$-planes on generic complete intersections was computed in \cite{bachmann2020a1euler}.

Wendt developes a Schubert calculus and computes Euler numbers in \cite{MR4071217}. In \cite{srinivasan2018arithmetic} Srinivasan and the second named author give an enriched count of lines meeting 4 general lines in $\PP^3$. 
 
Larson and Vogt count the bitangents to a smooth plane quartic curve \cite{larson2019enriched}. The relevant vector bundle is not relatively orientable. They introduce notion of \emph{relative orientability relative to a divisor} and show that the count of bitangents to a smooth plane quartic curve relative to a 'fixed line at infinity' is $16\langle1\rangle+12\langle-1\rangle$.

McKean proves an enriched version of B\'{e}zout's theorem in \cite{McKean}. 

The first named author computes several enriched Euler numbers in \cite{pauli2020computing} using Macaulay2.

\section{$\A^1$-degree of maps of smooth schemes}
\label{section: degree of maps btw schemes}
The content of the following sections is ongoing work by Jesse Kass, Marc Levine, Jake Solomon and the second named author.
\subsection{Motivation from Algebraic Topology}
\label{subsection: motivation for degree of maps}
Let $f:X\rightarrow Y$ be a map of compact, oriented $n$-manifolds without boundary with $Y$ connected.
Algebraic topology defines the \emph{degree} of $f$ (\cite[Chapter 5]{MR0226651}) to be $f_*[X]=\deg f \cdot [Y]$. This degree can again be expressed as the sum of local degrees 
\[\deg f=\sum_{q\in f^{-1}(p)}\deg_qf\]
where $\deg_qf$ is defined in the same way as before \eqref{equation: def jacobian},
that is, for oriented coordinates $x_1,\dots,x_n$ of $X$, the map $f$ is locally given by $f=(f_1,\dots,f_n):\R^n\rightarrow \R^n$ and
\[
\deg_q(f)=\begin{cases*}
      +1 & if Jf(q) \text{$>$} 0 \\
      -1        & if Jf(q) \text{$<$} 0
    \end{cases*}
\]
for $Jf=\det\frac{\partial
f_i}{\partial x_j}$.
\subsection{$\A^1$-degree}
\label{subsection: A1-degree of maps}
We want to construct a $\GW(k)$-valued degree for a map $f:X\rightarrow Y$ of smooth, proper $k$-schemes as a sum of local degrees
\[\deg^{\A^1} f:=\sum_{q\in f^{-1}(p)}\deg^{\A^1}_qf.\]
In order to do this we need to answer the following questions.
\begin{enumerate}
\item What is $\deg^{\A^1}_qf$?
\item What orientation data do we need?
\item When do we have finite fibers?
\item Is $\deg^{\A^1} f$ independent of $p$?
\end{enumerate}
Here is one answer to the third question: If the differential $Tf:TX\rightarrow TY$ is invertible at some point, we can arrange to have finite fibers away from a codimension 2 subscheme of $Y$. We will be able to content ourselves with throwing away codimension $2$ subsets of $Y$ because $\GW$ extends to an unramified sheaf \cite{MR1670591} \cite{MR2934577}, meaning a section of $\GW$ over the complement of a codimension $2$ subset extends to a section over $Y$. 

Assume that $p$ is a $k$-point and $q\in f^{-1}(p)$. When $k(q)$ is separable over $k$, we can assume that $k=k(q)$ (otherwise we base change to $k(q)$ and take the trace, see Example \ref{example: local A1 degree separable field extension}). We want to define the local degree $\deg^{\A^1}_qf$ as before, that is as the $\A^1$-degree of
\begin{align}
\label{eq: local degree of a map}
\begin{split}
\PP^n_{k}/\PP^{n-1}_{k}\simeq T_qX/(T_qX-\{0\})\simeq U/(U-\{q\})\\
\xrightarrow{\bar{f}}Y/(Y- \{p\})\simeq T_pY/(T_pY-\{0\})\simeq \PP^n_{k}/\PP^{n-1}_{k}.
\end{split}
\end{align}

To make this well-defined, we need orientation data to fix the isomorphisms in \eqref{eq: local degree of a map}.

Let $Tf:TX\rightarrow TY$ be the induced map on tangent bundles which is an element of $\operatorname{Hom}(TX,f^*TY)(X)$. Hence, its determinant $Jf:=\det Tf$ is an element of $Jf\in \operatorname{Hom}(\det TX,\det f^*TY)(X)$. To define $\langle Jf(q)\rangle\in \GW(k(q))$, we only need $Jf(q)$ to be well-defined up to a square, that is, we need $Jf(q)$ to be well-defined in $k(q)^{\times}/(k(q)^{\times})^2$. So if we can identify $Jf$ as a section of a square of a line bundle, we are good.

\begin{definition}
The map $f:X\rightarrow Y$ is \emph{relatively oriented} by the data of a line bundle $L$ on $X$ and an isomorphism $L^{\otimes2}\cong\operatorname{Hom}(\det TX,f^*\det TY)$.
\end{definition}

\begin{remark}

For $f$ relatively oriented, we have $Jf\in L_q^{\otimes 2}$, so $Jf(q)\in k(q)^{\times}/(k(q)^{\times})^2$ and $Jf_q\in \glob_{X,q}^{\times}/(\glob_{X,q}^{\times})^2$. Thus if $Jf(q)\neq 0$, then
\[\deg^{\A^1}_qf=\Tr_{k(q)/k}\langle Jf(q)\rangle \in \GW(k).\]
\end{remark}
\begin{definition}
Bases of $T_pY$ and $T_qX$ are \emph{compatible} if the corresponding element of the fiber $\operatorname{Hom}(\det TX,f^*\det TY)(q)$ is a square $l(q)\otimes l(q)$ for some $l\in L_q$.
\end{definition}

Requiring compatible bases makes the degree of \eqref{eq: local degree of a map} well-defined: Two different choices of compatible bases of $T_qX$ and $T_pY$ correspond to two elements $l$ and $l'$ of $L_q$ such that the corresponding elements of the fiber $\operatorname{Hom}(\det TX,f^*\det TY)(q)$ equal to $l(q)\otimes l(q)$ and $l'(q)\otimes l'(q)$, respectively, and thus differ by a square, i.e., are equal in $k(q)^{\times}/(k(q)^{\times})^2$. 
So we have a definition for $\deg^{\A^1}_qf$ (this answers question 1) given a relative orientation of $f$ (this answers question 2). In fact, we may even content ourselves with a relative orientation of the restriction of $f$ to the inverse image of the complement of a closed subset of $Y$ of codimension at least $2$.

It remains to see when the degree of a map is independent of the choice of $p$ (question 4). 
\begin{example}
The degree of a map is not necessarily independent of $p$. Let $C$ be the elliptic curve $C = \C/\Z[i]$. Then $C$ has two components of real points, given by the points of $\C$ with imaginary component $0$ and with imaginary component $1/2$. The map $C\xrightarrow{\times 2}C$ has different degrees over the different real components.
\end{example}

However, whenever two points can be connected by an $\A^1$, the local degree at those points are equal because of Harder's theorem.

\begin{theorem}[Harder's theorem]
Families of bilinear forms over $\A^1$ are stably constant (see \cite[Theorem 3.13, Chapter VII]{MR2235330} and \cite[Lemma 30]{MR3909901}).
\end{theorem}

We recall the definition of $\A^1$-chain connectedness from \cite{MR2803793}.
\begin{definition}
A $k$-scheme $Y$  is \emph{$\A^1$-chain connected} if for any finitely generated separable field extension $L/k$ and any two $L$-points $x,y\in Y(L)$ there are $x=x_0,x_1\dots,x_{n-1},x_n=y\in Y(L)$ and $\gamma_i:\A^1_L\rightarrow Y$ with $\gamma_i(0)=x_{i-1}$ and $\gamma_i(1)=x_i$ for $i=1,\dots,n$.
\end{definition}
In other words, a $k$-scheme is $\A^1$-chain connected if any two $L$-points can be connected by chain of maps from $\A^1_L$.
\begin{theorem}
\label{thm: deg well defined}

Let $f:X\rightarrow Y$ be a proper map of smooth $d$-dimensional $k$-schemes, such that $Tf$ is invertible at some point. Assume further that $f$ is relatively orientable after removing a codimension 2 subset of $Y$ and that $Y$ is $\A^1$-chain connected with a $k$-point $y$. Then \[\sum_{x\in f^{-1}(y)}\deg^{\A^1}_xf\in \GW(k)\]
is independent of a generically chosen point $y$.
\end{theorem}

\begin{definition}
With the assumption in Theorem \ref{thm: deg well defined} we define the \emph{degree} of $f:X\rightarrow Y$ to be equal to
\[\deg^{\A^1} f:=\sum_{x\in f^{-1}(y)}\deg^{\A^1}_xf.\]
\end{definition}

$f$ is generically finite and \'etale by assumption. It follows that $Jf(\text{generic pt}) \neq 0$, and 

\begin{corollary}
$\deg^{\A^1} f = \Tr_{k(X)/k(Y)}\langle Jf(\text{generic pt})\rangle$. 
\end{corollary}

Note that priori, $ \Tr_{k(X)/k(Y)}\langle Jf(\text{generic pt})\rangle$ is in $\GW(k(Y))$. It is a consequence of the theory that it in fact is the image of a well-defined element of $\GW(k)$.

\begin{example}
Let $C=\{(z,y):y^2=p(z)\}$ be an elliptic curve and let 
$\pi:C\rightarrow \PP^1$ be defined by $(z,y)\mapsto z$. The curve $C$ is oriented by $TC^*\simeq \glob$ with $\frac{dz}{2y}$ corresponds to $1$ and $(T\PP^1)^*\simeq \glob(-1)^{\otimes 2}$ where $dz$ is a square. Since $\pi^* (dz)=2y\frac{dz}{2y}$ we have that $Jf(\text{generic point of }C)=2y$ and thus
\begin{align*}
\deg^{\A^1} \pi=\Tr_{k(C)/k(z)}\langle 2y\rangle 
\end{align*}
which is equal to the form given by the matrix
\[
\begin{bmatrix}
\Tr_{k(C)/k(z)}(2y) & \Tr_{k(C)/k(z)}(2)\\
\Tr_{k(C)/k(z)}(2) & \Tr_{k(C)/k(z)}(2/y)
\end{bmatrix} 
=
\begin{bmatrix}
0 & 4\\
4 & 0
\end{bmatrix} 
\]
which is the hyperbolic form $\langle1\rangle+\langle-1\rangle$.
\end{example}
\section{Counting rational curves}

\label{subsection: counting rational curves}

It is an ancient observation that there is one line passing through two points in the plane. Similarly, given 5 points, there is one conic passing through them. These generalize to the question: how many degree $d$ rational plane curves are there passing through a generic choice of $3d-1$ points? Over an algebraically closed field, a {\em degree $d$ rational curve} means a map $$ u:\PP^1\rightarrow \PP^2,$$ $$t\mapsto [u_0(t),u_1(t),u_2(t)]$$  where the $u_i$ are polynomials of degree $d$, and more generally the domain of $u$ can be a genus $0$ curve. Over the complex numbers, the number of such curves passing through $3d-1$ points does not depend on the generic choice of the points themselves. For some low values of $d$, the answers $N_d$ are listed in Table \ref{tab:1} \cite[p.1]{MR1722053}. For $d=3$, $N_d$ was known to Steiner in 1848. For $d=4$, Zeuthen computed $N_d$ in 1873, but it was not until the 1980's that $N_5$ was computed. Then around 1994, Kontsevich computed a recursive formula for all $N_d$ with a breakthrough connection to string theory.

\begin{table}
\caption{Counting rational curves}
\label{tab:1}       
\begin{tabular}{lll}
\hline\noalign{\smallskip}
 $d$& $3d-1$ & $N_d=$ number of rational curves \\
\noalign{\smallskip}\hline\noalign{\smallskip}
 1& 2 & 1  \\
 2& 5 & 1  \\
 3& 8 & 12 \\
 4& 11 & 620\\
 5& 14 & 87,304\\
 $\dots$&$\dots$&$\dots$\\
\noalign{\smallskip}\hline
\end{tabular}
\end{table}

If we wish to count {\em real} degree $d$ rational curves passing through $3d-1$ points, we should assume that the set of points is permuted by complex conjugation. Even then, the number of such curves can depend on the chosen points. For example, there can be $8$,$10$, or $12$ real degree $3$ rational curves through $8$ real points. Welschinger recovers ``invariance of number" by counting each curve with a $+1$ or $-1$ instead of counting all curves as adding $+1$ to the total count. 

The Welschinger sign is given as follows. A smooth degree $d$ plane curve has genus ${d -1 \choose 2}$, and it follows that a degree $d$ rational curve has ${d -1 \choose 2}$ nodes in its image. Assign the {\em mass} $1$ to the non-split node, $-1$ to the split node, and ignore the complex conjugate pairs of nodes. See the figure in Section \ref{subsection: A1 Milnor numbers}. Define the mass $m(u)$ to be the sum of the masses of the nodes in the image curve. Then the rational curve $u$ is counted with sign $(-1)^{m(u)}$.

\begin{theorem}
(Welschinger) Fix positive integers $d$, $n_1$ and $n_2$ such that $n_1+2n_2 = 3d-1$. For any generic choice of $n_1$ real and $n_2$ complex conjugate pairs of points in $\PP^2(\C)$, the sum $$W_{d,n_2} = \sum_{\substack{u \text{ degree } d \\ \text{ real rational curve} \\\text{ through the points}}} \prod_{p \text{ node}\text{ of }u}  (-1)^{m(p)}$$ is independent of the choice of points.
\end{theorem} 

For small values of $d$ and $n_2$, the values $W_{d,n_2}$ are given in Table \ref{realtab:1}, which is from \cite{MR2138469}.

\begin{table}
\caption{Counting real rational curves with Welschinger signs}
\label{realtab:1}       
\begin{tabular}{lll}
\hline\noalign{\smallskip}
 $d$& $n_2$ & $W_{d,n_2}=$ signed count of real rational curves \\
\noalign{\smallskip}\hline\noalign{\smallskip}
 1& $n_2$ & 1  \\
 2& $n_2$ & 1  \\
 3& $n_2$ & $8- 2 n_2$\\
 4& 0 &240\\
 4& 1 & 144\\
  4& 2 & 80\\
   4& 3 &40\\
    4& 4 & 16\\
      4& 5 & 0\\
 $\dots$&$\dots$&$\dots$\\
\noalign{\smallskip}\hline
\end{tabular}
\end{table}

Jake Solomon's thesis computes all of the $W_{d,n_2}$ recursively \cite{MR2717339} as the degree of a certain map. 

We want to do this over an arbitrary field $k$. For example, what about counting rational curves over $k = \FF_p,$ $\QQ_p$, or $\QQ$? 
\begin{definition}
A genus $g$, $n$-marked \emph{stable map} to $\PP^2$ consists of the data $(u:C\rightarrow \PP^2,p_1,\dots,p_n)$ where $C$ is a genus $g$ curve, $p_1,\dots,p_n\in C$ are smooth closed points of $C$ and $u$ is a morphism with only finitely many automorphisms.
Denote by
\begin{align*}\overline{\mathcal{M}_{\PP^2,n}}(0,d):=\{(u:C\rightarrow \PP^2, p_1,\dots,p_n): C \text{ rational, degree }d\text{ curve, }\\u \text{ stable and }p_i\in C\text{ smooth points}\}\end{align*}
 the \emph{Kontsevich moduli space}, that is the moduli space of genus $0$, degree $d$, $n$-marked stable maps into $\PP^2$.
\end{definition}
Consider the evaluation map
\[\operatorname{ev}:\overline{\mathcal{M}_{\PP^2,n}}(0,d)\rightarrow (\PP^2)^n, u\mapsto [u(p_1),\dots,u(p_n)].\] Its fiber over a $k$-point is precisely those rational curves through the $k$-points $p_1,\ldots,p_n$. So if we are capable of defining the degree of $\ev$ we will obtain a weighted count in $\GW(k)$ which does not depend on the choice of points, as long as they are chosen generically. We want to allow points to move in Galois conjugate orbits as in Welschinger's theorem. Fix a Galois action on the points by $\sigma:\operatorname{Gal}(\bar{k}/k)\rightarrow \Sigma_n$. This is equivalent to choosing the residue fields of the points. For example, over $\QQ$ for $d \geq 4$, we could choose a conjugate pair of points over $\QQ(\sqrt{2})$, six points in a single orbit defined over the splitting field of $x^3 - 7$ and the rest $\QQ$ points. Given $\sigma$, the set of these residue fields is given by \begin{equation}\label{resfieldssigma}\{\kbar^{\operatorname{stab}{o}}: o \text{ orbit of }\sigma \}.\end{equation}Here, $\bar{k}^{\operatorname{stab}o}$ is the field fixed by the stabilizer $\operatorname{stab}o$. 

Given this data, we can twist $\overline{\mathcal{M}_{0,n}}$, $(\PP^2)^n$ and the evaluation map so that points with these residue fields correspond to a rational point of the twist of $(\PP^2)^n$. In a little more detail, $\sigma$ determines an action on  $(\PP_{\kbar}^2)^n$ by permutation which combines with the standard Galois action to form a twisted action $g(p_1,\ldots, p_n) = (g p_{\sigma^{-1}(1)}, \ldots, g p_{\sigma^{-1}(n)})$. Taking the invariants of this twisted action defines a $k$-scheme $(\PP^2)^n_\sigma$, which can be described as a restriction of scalars $$(\PP^2)^n_\sigma \cong \prod_{ o \text{ orbit of }\sigma}\operatorname{Res}^{\kbar^{\operatorname{stab}{o}}}_k \PP^2 .$$ 

One can also twist $\overline{\mathcal{M}_{0,n}}(\PP^2,d)$ and the evaluation map, resulting in a  twisted evaluation map $\operatorname{ev}_{\sigma}:\overline{\mathcal{M}_{0,n}}(\PP^2,d)_{\sigma}\rightarrow (\PP^2)_{\sigma}^n$. A collection $p_1,\ldots,p_n$ of Galois conjugate orbits with residue fields compatible with $\sigma$ (i.e. whose residue fields are \eqref{resfieldssigma}) is a rational point of $(\PP^2)_{\sigma}^n$. The fiber of $\operatorname{ev}_{\sigma}$ over such a point consists of the rational curves passing through the $p_i$.

With considerable work, it can be shown that $\operatorname{ev}_{\sigma}$ satisfies the hypothesis of Theorem \ref{thm: deg well defined} after removing a codimension $2$ subset of $(\PP^2)_{\sigma}^n$ and its preimage. It follows that there is a well-defined degree $\deg^{\A^1} \operatorname{ev}\in \GW(k)$.
By construction, this degree  $\deg^{\A^1} \operatorname{ev}\in \GW(k)$ is a sum over rational curves passing through the rational points $(p_1,\dots,p_n)$.

Let $N_{d,\sigma}:=\deg^{\A^1} \operatorname{ev}_{\sigma}$ giving the enriched rational curve count. A natural question to ask at this point is: 

\begin{question}
What are the local degrees $\deg^{\A^1}_u \operatorname{ev}_{\sigma}$ at a rational curve $u$? 
\end{question}

The answer has a geometric interpretation. The set of nodes of $u(C)$ are defined over $k(u)$. (An individual node could have a larger field of definition, but then it would come in a Galois orbit.) The tangent directions at these nodes (i.e. the lines of the tangent cones) determine a field extension $k(u) \subseteq L(u)$. The discriminant $\disc(L(u)/k(u)) \in k(u)^*/(k(u)^*)^2$ of the extension $k(u) \subseteq L(u)$ is the discriminant of the transfer $\Tr_{L(u)/k(u)}\langle 1 \rangle$, or in other words the determinant of a Gram matrix corresponding to this form. We have $$ \deg^{\A^1}_u \operatorname{ev}_{\sigma} = \Tr_{k(u)/k} \disc(L(u)/k(u)).$$ 

We can match this up with Welschinger's theorem by interpreting $$\prod_{p \in \Nodes(u)} \type(p) \langle -1 \rangle$$ as an element of $\GW(k(u))$. While this is an abuse of notation, as $\type(p)$ may lie in a larger field, after taking the product with the types of the Galois conjugates, we arrive at the norm and an element of $\GW(k(u))$. Comparing the definition of the type with the discriminant, we wee that $\disc(L(u)/k(u)) = \prod_{p \in \Nodes(u)} \type(p) \langle -1 \rangle$, with our particular definition of the type. To make this prettier, define the {\em mass} $m(p)$ of a node $p$ by $$m(p) = \type(p) \langle -1 \rangle .$$ Combining the above, we obtain the following generalization of Welschinger's theorem:

\begin{theorem}
(Kass--Levine--Solomon--W.) Let $k$ be a field of characteristic not $2$ or $3$. Let $d \geq 1$, and fix the data of the field extensions in a Galois stable set of $3d-1$ points over the algebraic closure. We use a permutation representation $\sigma:\operatorname{Gal}(\bar{k}/k)\rightarrow \Sigma_n$ for this. Then for any generic points $p_1,\ldots,p_{3d-1}$ of $\PP^2(\kbar)$ permuted by $\sigma$, we have the equality in $\GW(k)$ $$N_{\sigma, d} = \sum_{\substack{u \text{ degree } d \\ \text{ rational curve} \\\text{ through } p_1, \ldots p_n}}\Tr_{k(u)/k} \prod_{p \text{ node}\text{ of }u}  m(p)$$
\end{theorem}

\begin{remark}
Note that $N_{\sigma, d}$ only depends on the field extension types of the points $p_i$. When the degrees of all these field extensions is $\leq 3$, M. Levine showed this in \cite[Example 3.9]{Levine-Welschinger}. 
\end{remark}

We end with some small examples.

\begin{example}
If $d=1,2$ we get $N_{d,\sigma}=\langle 1\rangle$ for all $\sigma$.
\end{example}

\begin{example}
Let $d=3$, and suppose the field extension types of points permuted by $\sigma$ are all separable over $k$. Using M. Levine's count of 1-nodal curves in a pencil \cite[Corollary 12.4]{levine2017enumerative}, one can compute
\[N_{d,\sigma}=2(\langle 1\rangle+\langle -1\rangle)+\Tr_{k(\sigma)/k}\langle 1\rangle\]
where $\Tr_{k(\sigma)/k}$ is the sum of the trace forms
\[\sum_{x\in  \operatorname{orbit}(\sigma)}\Tr_{\bar{k}^{\operatorname{stab}x}/k}\langle1\rangle.\] Equivalently, if $D=(p_1,\dots,p_n)$ is a divisor on $C$ over $k$ permuted as in $\sigma$, such that all $p_i$ are distinct and defined over separable field extensions, then $D$ defines a finite \'{e}tale algebra over $k$ and
\[\Tr_{k(\sigma)/k}\langle 1\rangle=\text{trace form of }D.\]
\end{example}

\begin{remark}
$N_{d,\sigma}$ is not just $\langle1\rangle's$ and $\langle-1\rangle's$, and depends on $\sigma$.
\end{remark}

\begin{table}
\caption{$\GW(k)$-enriched counts of rational curves}
\label{tab:3}       
\begin{tabular}{lll}
\hline\noalign{\smallskip}
 $d$& $\sigma$ & $N_{d,\sigma}=$ count of rational curves \\
\noalign{\smallskip}\hline\noalign{\smallskip}
 1 & all $\sigma$ & $\langle 1 \rangle$  \\
 2 & all $\sigma$ & $\langle 1 \rangle$  \\
 3& $\sigma$ corresponding to separable $k \subseteq k(p_i)$ &$ 2(\langle 1\rangle+\langle -1\rangle)+\Tr_{k(\sigma)/k}\langle 1\rangle $\\
  $\dots$&$\dots$&$\dots$\\
\noalign{\smallskip}\hline
\end{tabular}
\end{table}

\section*{Ackknowledgements}
Kirsten Wickelgren was partially supported by NSF CAREER grant DMS-2001890. Sabrina Pauli gratefully acknowledges support by the RCN Frontier Research Group Project no.  250399 “Motivic Hopf Equations.” We also wish to thank Joe Rabinoff. 

%
%

\bibliographystyle{spmpsci}      
\bibliography{PIMSbib}   

%
%

\end{document}